\documentclass[a4paper,10pt]{article}
\usepackage[english]{babel}
\usepackage[utf8]{inputenc}
\usepackage[T1]{fontenc}
\usepackage{textcomp} 
\usepackage{amsmath}
\usepackage{amsfonts}
\usepackage{amsthm}
\usepackage{mathrsfs}
\usepackage{amssymb}
\usepackage{enumerate}
\usepackage{csquotes}
\usepackage{bbm}
\usepackage{graphicx}
\usepackage[mathscr]{euscript}
\usepackage{clrscode3e}
\usepackage{color}

\DeclareSymbolFont{rsfs}{U}{rsfs}{m}{n}
\DeclareSymbolFontAlphabet{\mathscrsfs}{rsfs}

\theoremstyle{definition}
\newtheorem{Def}{Definition}[section]

\newtheorem{Rmk}[Def]{Remark}

\theoremstyle{plain}

\newtheorem{Thm}[Def]{Theorem}
\newtheorem{Lemma}[Def]{Lemma}

\numberwithin{equation}{section}

\newcommand{\R}{\mathbb{R}}
\newcommand{\E}{\mathbb{E}}

\newcommand{\Z}{\mathbb{Z}}
\newcommand{\N}{\mathbb{N}}

\usepackage{scalerel,stackengine}
\stackMath
\newcommand\reallywidehat[1]{%
\savestack{\tmpbox}{\stretchto{%
  \scaleto{%
    \scalerel*[\widthof{\ensuremath{#1}}]{\kern-.6pt\bigwedge\kern-.6pt}%
    {\rule[-\textheight/2]{1ex}{\textheight}}%WIDTH-LIMITED BIG WEDGE
  }{\textheight}% 
}{0.5ex}}%
\stackon[1pt]{#1}{\tmpbox}%
}
\newcommand{\footremember}[2]{%
   \footnote{#2}
    \newcounter{#1}
    \setcounter{#1}{\value{footnote}}%
}

\makeatletter

\def\oversortoftilde#1{\mathop{\vbox{\m@th\ialign{##\crcr\noalign{\kern3\p@}%
      \sortoftildefill\crcr\noalign{\kern3\p@\nointerlineskip}%
      $\hfil\displaystyle{#1}\hfil$\crcr}}}\limits}

\def\sortoftildefill{$\m@th \setbox\z@\hbox{$\braceld$}%
  \braceld\leaders\vrule \@height\ht\z@ \@depth\z@\hfill\braceru$}

\makeatother

\title{Numerical simulation of Generalized Hermite Processes}
\author{A Ayache \footremember{Lille}{Univ. Lille, 
CNRS, UMR 8524 - Laboratoire Paul Painlev\'e, F-59000 Lille, France. antoine.ayache@univ-lille.fr}, J. Hamonier \footnote{Univ. Lille, CHU Lille, ULR 2694 - METRICS : \'Evaluation des technologies de sant\'e et des
pratiques m\'edicales, F-59000 Lille, France. julien.hamonier@univ-lille.fr}, L. Loosveldt\footnote{ \underline{\textbf{Corresponding author:}} Université de Liège, Département de Mathématique -- zone Polytech 1, 12 allée de la Découverte, B\^at. B37, B-4000 Liège. l.loosveldt@uliege.be}}

\begin{document}

\maketitle

\begin{abstract}
Hermite processes are paradigmatic examples of stochastic processes which can belong to any Wiener chaos of an arbitrary order; the well-known fractional Brownian motion belonging to the Gaussian first order Wiener chaos and the Rosenblatt process belonging to the non-Gaussian second order Wiener chaos are two particular cases of them. Except these two particular cases no simulation method for sample paths of Hermite processes is available so far. The goal of our article is to introduce a new method which potentially allows to simulate sample paths of any Hermite process and even those of any generalized Hermite process.
% devoted to this end and to obtain an estimate of its almost sure rate of convergence in the uniform norm. to test it through simulations of Hermite processes in the third order non-Gaussian Wiener chaos; this algorithm also allows to simulate sample paths of generalized Hermite processes. 
 Our starting point is the representation for the latter process as random wavelet-type series, obtained in our very recent paper \cite{ayachehamonierloosveldt}. We construct from it a "concrete" sequence of piecewise linear continuous random functions which almost surely approximate sample paths of this process for the uniform norm on any compact interval, and we provide an almost sure estimate of the approximation error. Then, for the Rosenblatt process and more importantly for the third order Hermite process, we propose algorithms allowing to implement this sequence and we illustrate them by several simulations. Python routines implementing these synthesis procedures are available upon request.
 
%The goal of this paper is to introduce a new algorithm designed to simulate trajectories of any generalized Hermite process for which no simulation is available so far. This article is motivated by \cite{JSSv005i07,PIPIRAS200549,ABRY20062326}. The two first papers give various methods to simulate paths of Hermite process of order 1,  usually called fractional Brownian motion. In the second, the authors provide two methods to simulate paths of Hermite process of order 2, usually called Rosenblatt process, and set the problem for Hermite process of order greater than 3. Our starting point is the representation for generalized Hermite process as random wavelet-type series obtained in \cite{ayachehamonierloosveldt}. As these series can not be directly used for numerical computations, we first derive from it a numerically computable piecewise constant random walk. Next, we do a linear interpolation of the different jump times to obtain a more concrete approximation of the trajectories of this type of process.  Finally, we determine their a.s. rates of convergence in the space of continuous functions. In the end, the representations of Rosenblatt process and Hermite process of order 3 illustrate our method. We hope that our work could be used for simulations of natural phenomena as well as for statistical inference.

%\cite{JSSv005i07}, \cite{ABRY20062326}, \cite{torrestudor09}
\end{abstract}


\section{Introduction and motivation}

For many years, there has been growing interest in modelling many real-life situations using fractional stochastic processes. Among other things, these families of parametri\-zed processes offer a natural framework for simulating phenomena with short- or long- range dependence properties, as well as self-similarity. We refer e.g. to the monographs \cite{doukhan2002theory,Embrechts2002} for a comprehensive view on both properties. A standard strategy consists in using the famous fractional Brownian motion \cite{mandelbrot1968fractional}. This process has been for instance used for simulations in astronomy \cite{Pan2007}, biology \cite{Ecksteina}, climatology \cite{Franzke2020}, image processing \cite{McGarrya}, internet traffic modelling \cite{Norros1995} and physics \cite{Metzler2014}. The fractional Brownian motion of Hurst paramter $H \in (0,1)$ on a probability space $(\Omega,\mathcal{F},\mathbb{P})$ is known to be the unique centred Gaussian process $\{B^H_t\}_{t \geq 0}$ with covariance function given, for all, $s,t \geq 0$, by
\begin{equation}\label{eqn:cov}
 \mathbb{E}[B_t^H B_s^H] = \frac{1}{2}(|t|^{2H}+|s|^{2H}-|t-s|^{2H}).
\end{equation}
To numerically simulate a fractional Brownian motion, one can use a wavelet-based synthesis, which relies on its wavelet-type random series representation due to Meyer, Sellan and Taqqu \cite{MR1755100}. We refer to the papers \cite{Abry1996,JSSv005i07} for details and practical features about the implementation of such an algorithm.

In many situations, the Gaussianity of the fractional Brownian motion is a too strong assumption to model a phenomenon of interest. Such a case arises, for instance, in mathematical finance \cite{Fauth2016,STOYANOV2019}, hydrology \cite{Taqqu1978} or internet traffic modelling \cite{VivekKumarSehgal_2019}. To drop the Gaussian assumption, one can use the affiliation of fractional Brownian motion in the family of Hermite processes.

For $d \geq 1$ and $H \in (1/2,1)$, the Hermite process of order $d$ and Hurst parameter $H$ is defined, at any time $t\in\R_+$, by the multiple Wiener integral
\begin{equation}\label{eqn:def:ghp}
\sqrt{\frac{H(2H-1)}{d! \beta\left(\frac{1}{2}-\frac{1-H}{d}, \frac{2-2H}{d}\right)}}\int_{\R^d}' \left( \int_0^t \prod_{j=1}^d (s-x_\ell)_+^{\frac{H-1}{d}-\frac{1}{2}} \, ds\right) dB(x_1) \cdots dB(x_d),
\end{equation}
where $\{B(x)\}_{x \in \R}$ is a usual Brownian motion, $\beta$ denotes the classical Euler's Beta function \footnote{for all $x,y>0$, $\beta(x,y)=\int_0^1 t^{x-1}(1-t)^{y-1}dt$} and we use the convention that, for all $(x,\alpha) \in \R^2$,
\begin{align*}
x_+^\alpha:=\left\{ \begin{array}{cl}
                       x^\alpha & \text{if } x > 0\\
                       0 & \text{otherwise.}
                     \end{array} \right.
\end{align*}
Observe that the symbol $\int_{\R^d}'$ in \eqref{eqn:def:ghp} denotes integration over $\R^d$ with diagonals $\{x_\ell=x_{\ell'}\}$, $\ell \neq \ell'$, excluded. It is the paradigmatic example of a stochastic process in the Wiener chaos of order $d$. When $d=1$, it reduces to the fractional Brownian motion, which is the only Gaussian process in this family. The Hermite process of any order is $H$-self-similar, has stationary increments, exhibits long-range dependence and has the same covariance function \eqref{eqn:cov} as the fractional Brownian motion. In particular, its sample paths are, on each compact interval, Hölder continuous of order $\delta$ for every $\delta \in (0,H)$.

The Hermite process of order $2$ is also known as the Rosenblatt process. In \cite{ABRY20062326}, the authors propose two practical methods for simulating sample paths of it. They are based on the random wavelet-type series representation of this same process \cite{pipiras2004wavelet}. For the first method, the authors obtain a uniform convergence result \cite[eq 2.5]{ABRY20062326}. Whereas for the second method, the authors formulate a conjecture concerning the uniform convergence on any compact set of their representation \cite[eq 2.10]{ABRY20062326}. Note also that the authors in \cite{torrestudor09} propose an algorithm based on a Donsker-type theorem to numerically approximate the Rosenblatt process by perturbed random walks. Nevertheless, the rate of convergence of this approximation procedure is not provided.

In \cite{pipiras2004wavelet}, the author raises the question of whether the wavelet-type expansion he obtains for the Rosenblatt process can be extented to any Hermite process. This problem remained open up to our very recent work \cite{ayachehamonierloosveldt}. In this last paper, we even consider a larger family of processes which contains the Hermite ones (see Definition \ref{def:genhp} below). We provide a wavelet-type expansion for any process in this family which has the advantage to clearly separate the low and high frequency parts of the process. Moreover, we explicit the rate of convergence of this low frequency part towards its corresponding process.

 Despite the growing interest for Hermite processes in the literature, up to now, there is no method to numerically simulate sample paths of such a process, as soon as the order $d \geq 3$; the latter fact is mentioned on page 43  of the very recent book \cite{tudor2023}. The aim of the current paper is to make a breakthrough in this area in order to overcome the serious drawback due to the lack of simulation methods. Our starting point is the representation of any arbitrary generalized Hermite process as random wavelet-type series, obtained in our very recent paper \cite{ayachehamonierloosveldt}. We construct from it a "concrete" sequence of piecewise linear continuous random functions which almost surely approximate sample paths of this process for the uniform norm on any compact interval, and we provide an almost sure estimate of the approximation error. Then, for the Rosenblatt process and more importantly for the third order Hermite process, we propose algorithms allowing to implement this sequence and we illustrate them by several simulations.
 
% More precisely, we show that any generalized Hermite process can almost surely be approximated, for the uniform on any compact interval, by an explicit sequence of wavelet-based random walk which can be numerically simulated. Also, we provide an estimate of the almost sure approximation error.
% %The aim of the current paper is to overcome this serious drawback by presenting, 
% for any generalized Hermite process, a sequence of wavelet-based random walk which can be numerically simulated and which converges to its corresponding process. Moreover, we also provide the rate of this convergence, by mean of sup-norm on any compact subset of $[0,\infty)$. 

% Note that this simulation procedure is different from the ones presented in \cite{Abry1996,JSSv005i07} and \cite{ABRY20062326} for the fractional Brownian motion and the Rosenblatt process respectively.

We believe that our work could be of potential interest for applications for which Hermite processes are underlying models. Among other things, it can be used to test numerical efficiency of various statistical estimators for the Hurst parameter of Hermite processes, based, for instance, on quadratic variation \cite{Tudor2009,Chronopoulou2011,10.3150/23-BEJ1627}, (approached) maximum (log-)likelihood \cite{Whittle1962,Dahlhaus1989,Robinson1995,Bardet2014}, log-periodogram \cite{Moulines2003} or wavelet variation \cite{Bardet2010,Clausel2014,ORBi-4a2ba0cf-5d65-4a71-99c8-62f449b47bec}. %We also hope to open the door for statistical inference on the order $d$ of the process. 

The rest of the paper is organized as follows. In Section \ref{sec:statement}, we first precisely define the class of generalized Hermite processes, introduced in \cite{MR3163219}, as well as their wavelet-type random series representation obtained in our very recent article  \cite{ayachehamonierloosveldt}, then we state our two main results which are issued from the latter representation. Section \ref{sec:proofs} is devoted to their proofs. In Section \ref{sec:examples}, we provide algorithms to numerically simulate the Rosenblatt process and more importantly the Hermite process of order $3$, as well as the generalized Hermite process of order $3$. Then we test these algorithms though several simulations.

%Then, after recalling some notations and the main result from the article \cite{ayachehamonierloosveldt}, we derive from it a sequence of computable piecewise constant random walks which converges to a given generalized Hermite process. Next, we do a linear interpolation of the different instant of jump to obtain a more concrete realization of the trajectories of this type of process. In Section \ref{sec:proofs}, we show that both the piecewise constant random walk and its interpolation are almost surely uniformly convergent to their corresponding process. Moreover, we explicit this rate of convergence. As an application, in Section \ref{sec:examples}, we provide algorithms to numerically simulate the Rosenblatt process and the Hermite process of order $3$, then we test these algorithms though various simulations.

\section{Background and statement of the main results} \label{sec:statement}

The family of processes we are interested in, called generalized Hermite processes, was introduced in \cite{MR3163219}. These processes are self-similar and have stationary increments. They include Rosenblatt and any other Hermite process \cite[Chap. 4]{Pipiras_Taqqu_2017}.

\begin{Def}\label{def:genhp}
The generalized Hermite process of any arbitrary order $d\ge 2$, denoted by $\{X_{\mathbf{h}}^{(d)}(t)\}_{t \in \R_+}$, depends on a vector-valued Hurst parameter $\mathbf{h}:=(h_1,\ldots, h_d)$ whose coordinates satisfy 
\begin{equation}\label{eqn:hermi:cond}
h_1,\cdots,h_d \in (1/2,1) \text{ and } \sum_{\ell=1}^d h_\ell > d-\frac{1}{2}.
\end{equation}
This process belongs to the non-Gaussian $d$th Wiener chaos, since it is defined, for each $t \in \R_+$, through the multiple Wiener integral:
\begin{equation}
X^{(d)}_{\mathbf{h}}(t) := \int_{\R^d}' K^{(d)}_{\mathbf{h}}(t,x_1,\ldots,x_d) dB(x_1) \cdots dB(x_d),
\end{equation}
where the deterministic kernel function $K_{\mathbf{h}}$ is given, for every $(t,x_1,\ldots,x_d) \in \R_+ \times \R^d$, by
\begin{equation}\label{eqn:def:kernel}
K^{(d)}_{\mathbf{h}}(t,x_1,\cdots,x_d) := \frac{1}{\prod_{\ell=1}^d \Gamma(h_\ell-1/2)} \int_0^t \prod_{j=1}^d (s-x_\ell)_+^{h_\ell-3/2} \, ds.
\end{equation}
\end{Def}

\begin{Rmk}\label{rmk:holder}
As noted in \cite[Theorem 3.5.]{MR3163219}, the process $\{X_{\mathbf{h}}^{(d)}(t)\}_{t \in \R_+}$ has stationary increments and is self-similar with self-similarity index $H\in (1/2,1)$ given by 
\begin{equation}
\label{eq:add2}
H:=\sum_{\ell=1}^d h_\ell -d +1. 
\end{equation}
In particular, for any $s,t \geq 0$ and $p>0$, one has
\[ \E[|X_{\mathbf{h}}^{(d)}(t)-X_{\mathbf{h}}^{(d)}(s)|^p] = |t-s|^{Hp}\, \E[|X_{\mathbf{h}}^{(d)}(1)|^p].\]
As the hypercontractivity property for multiple integrals (see e.g.\cite[Theorem 2.7.2]{MR2962301}) entails that, for all such $p$, $\E[|X_{\mathbf{h}}^{(d)}(1)|^p]< \infty$, it follows from Kolmogorov continuity theorem that there exists a modification of $\{X_{\mathbf{h}}^{(d)}(t)\}_{t \in \R_+}$ and $\Omega_1$, an event of probability $1$, such that, on $\Omega_1$, sample paths of this modification are locally (that is on each compact interval) $\gamma$-Hölder continuous functions, for any $\gamma \in (0,H)$. All along this paper, we identify $\{X_{\mathbf{h}}^{(d)}(t)\}_{t \in \R_+}$ with this modification. In this case, it means that, for any $T>0$ and $\gamma \in (0,H)$, there is a positive finite random variable $C_{T,\gamma}$ such that, almost surely, for any $s,t \in [0,T]$
\begin{equation}\label{eqn:holderrmk}
|X_{\mathbf{h}}^{(d)}(t)-X_{\mathbf{h}}^{(d)}(s)| \leq C_{T,\gamma} |s-t|^\gamma.
\end{equation}
\end{Rmk}

\begin{Rmk}\label{rmk:normalisation}
Observe that when all the coordinates $h_1,\ldots, h_d$ of the vector-valued parameter $\mathbf{h}$ are equal to $1+\frac{H-1}{d}$, then the process $\{X_{\mathbf{h}}^{(d)}(t)\}_{t \in \R_+}$ reduces to the usual Hermite process of Hurst parameter $H$ and order $d$, up to a r- normalisation. The constant $\sqrt{\frac{H(2H-1)}{d! \beta\left(\frac{1}{2}-\frac{1-H}{d}, \frac{2-2H}{d}\right)}}$ used in \eqref{eqn:def:ghp} is chosen such that the process at time $1$ has unit variance. On the other side, the constant $\frac{1}{\prod_{\ell=1}^d \Gamma(h_\ell-1/2)}$ used in \eqref{eqn:def:kernel} is more practicable for computation in the general context of Definition \ref{def:genhp}.
\end{Rmk}

Before stating our main results, let us recall some notations used in \cite{ayachehamonierloosveldt}. %Indeed, in this paper, we present a wavelet-type expansion for generalized Hermite processes with it's rate of convergence. We need to introduce some notations. 
First, $\phi$ and $\psi$ stand for, respectively, the univariate scaling function and mother wavelet associated with a Meyer orthonormal wavelet basis of $L^2(\R)$. Let us recall that $\phi$ and $\psi$ belong to the Schwartz class $\mathcal{S}(\R)$ of infinitely differentiable functions whose derivatives of any order rapidly decay at infinity. Moreover, the Fourier transforms $\widehat{\phi}$ and $\widehat{\psi}$ are infinitely differentiable compactly supported functions satisfying
\[
{\rm supp}\,\widehat{\phi}\subseteq \left [-\frac{4\pi}{3},\frac{4\pi}{3}\right]\quad\mbox{and}\quad {\rm supp}\,\widehat{\psi}\subseteq \left [-\frac{8\pi}{3},\frac{8\pi}{3}\right]\setminus \left(-\frac{2\pi}{3},\frac{2\pi}{3}\right).
\]

The wavelet-type expansion of generalized Hermite processes relies on so-called fractional primitives of the mother wavelet $\psi$.

\begin{Def}\label{def:fractwav}
For all $h \in [1/2,+\infty)$ (resp. $h\in (-\infty,1/2)$), the fractional primitive of $\psi$ of order $h-1/2$ (resp. the fractional derivative of $\psi$ of order $1/2-h$) is the function $\psi_h\in\mathcal{S}(\R)$ defined through its Fourier transform by:
\[ \widehat{\psi_h}(0)=0 \text{ and, for all $\xi \neq 0$ } \widehat{\psi_h}(\xi)=(i \xi)^{1/2-h}\widehat{\psi}(\xi).\]
\end{Def} 

Note that, since the compactly supported Fourier transform $\widehat{\psi}$ vanishes in a neighbourhood of $0$, fractional primitives and derivatives of a univariate Meyer mother wavelet $\psi$ belong to the Schwartz class $\mathcal{S}(\R)$. At the opposite, since $\widehat{\phi}(0)=1\ne 0$, fractional primitives and derivatives of a univariate Meyer scaling function $\phi$ fail to be smooth well-localized functions. In order to overcome this serious difficulty, a clever idea of \cite{MR1755100} was to "replace" fractional primitive or derivative of $\phi$ by the so called fractional scaling function $\Phi_\Delta^{(\delta)}$, which belongs to $\mathcal{S}(\R)$ and which was defined in \cite{MR1755100} as follows:

\begin{Def}\label{def:fractscal}
The \textit{fractional scaling function} of order $\delta \in (-\infty,+\infty)$ of a univariate Meyer scaling function $\phi$ is the function $\Phi_\Delta^{(\delta)}\in\mathcal{S}(\R)$ defined through its Fourier transform by:
\[ \widehat{\Phi}_\Delta^{(\delta)}(\xi) = \left( \frac{1-e^{-i \xi}}{i \xi} \right)^\delta \widehat{\phi}(\xi)~\forall \, \xi \neq 0 \text{ and }  \widehat{\Phi}_\Delta^{(\delta)}(0)=1.\]
Notice that, similarly to $\widehat{\phi}$, the function $\widehat{\Phi}_\Delta^{(\delta)}$ has a compact support satisfying 
\begin{equation}\label{eqn:supportfractscalfunction}
{\rm supp}\,\widehat{\Phi}_\Delta^{(\delta)}\subseteq \left [-\frac{4\pi}{3},\frac{4\pi}{3}\right].
\end{equation}
\end{Def}
One can check, from elementary properties of the Fourier transform on the space $\mathcal{S}(\R)$ (see e.g. the seminal book \cite{schwartz:78}) that $\Phi_\Delta^{(\delta)}$ and $\psi_h$ are well-defined functions belonging to $\mathcal{S}(\R)$, which means that they are infinitely differentiable functions satisfying, for all $m \in \N_0$ and $L>0$, 
\begin{equation}\label{maj:psih}
\sup_{x\in\R} \left\lbrace (3+|x|)^L \left (\Big | \frac{d^m}{dx^m}\Phi_\Delta^{(\delta)}(x)\Big|+\Big | \frac{d^m}{dx^m} \psi_h(x)\Big| \right) \right\rbrace < +\infty. 
\end{equation}

An other fundamental element of the wavelet-type expansion of generalized Hermite expansion is a family of generalized FARIMA sequences associated to the chaotic random variables $\mu_{J, \mathbf{k}}$ defined, for all $J \in \Z$ and $\mathbf{k} \in \Z^d$ by
\[ \mu_{J,\mathbf{k}} := 2^{J \frac{d}{2}}\int_{\R^d}' \phi(2^J x_1-k_1) \cdots \phi(2^J x_d-k_d) \, dB(x_1) \ldots dB(x_d).\]
For all $\delta \in (-1/2,1/2)$  we set $\gamma_0^{(\delta)}=1$, and for each $p \in \N$,
\[ \gamma_p^{(\delta)} = \frac{\delta \, \Gamma(p+\delta)}{\Gamma(p+1) \Gamma(\delta+1)}.\]
Then, for all $J \in \Z$, the generalized FARIMA sequence $(\sigma_{J, \mathbf{k}}^{(\mathbf{h})})_{\mathbf{k} \in \Z^d}$ is defined by
\begin{equation}\label{eq:ext-farima}
\sigma_{J,\textbf{k}}^{(\mathbf{h})}:=\sum_{\textbf{p}\in\N_0^d}\Big (\prod_{l=1}^d \gamma_{p_l}^{(h_l-1/2)}\Big)\mu_{J,\mathbf{k}-\mathbf{p}}.
\end{equation}
 Let us consider, for $J \in \Z$, the sequence $(g_{J,k}^\phi)_{ k \in \Z}$ of i.i.d. $\mathcal{N}(0,1)$ Gaussian random variable defined, for all $k \in \Z$, by
\[ g_{J,k}^\phi := 2^{J/2} \int_{\R} \phi(2^Jx-k) \, dB(x).\]
For such a $J$ and $\delta \in (-1/2,1/2)$, the usual Gaussian FARIMA $(0,\delta,0)$ sequence $(Z_{J,\ell}^{(\delta)})_{\ell \in \Z}$ associated to $(g_{J,k}^\phi)_{ k \in \Z}$ is given, for all $\ell \in \Z$, by
\[ Z_{J,\ell}^{(\delta)}:= g_{J,\ell}^\phi  + \sum_{p=1}^{+ \infty} \gamma_p^{(\delta)} g_{J,\ell-p}^\phi. \]
Then, we can use the more explicit expression \cite[Proposition 2.7]{ayachehamonierloosveldt}, for all $J \in \Z$ and $\mathbf{k} \in \Z^d$,
\begin{align}\label{eqn:defofsigma}
\sigma_{J,\textbf{k}}^{(\mathbf{h})} = \sum_{m=0}^{\lfloor d/2 \rfloor } (-1)^m \sum_{P \in \mathcal{P}_m^{(d)}} \prod_{r=1}^m \mathbb{E}[ Z_{J,k_{\ell_r}}^{(h_{\ell_r}-1/2)}Z_{J,k_{\ell_r'}}^{(h_{\ell_r'}-1/2)}] \prod_{s=m+1}^{d-m} Z_{J,k_{\ell_{s}''}}^{(h_{\ell_{s}''}-1/2)},
\end{align}
where $\mathcal{P}_m^{(d)}$ is the finite set of all partitions of $\{1,\dots,d\}$ with $m$ (non ordered) pairs and $d-2m$ singletons and the indices $\ell_r$, $\ell_r'$ and $\ell_{s}''$ are such that 
\[ P =\Big\{\{\ell_1,\ell_1 '\},\ldots,\{\ell_m,\ell_m '\},\{\ell_{m+1}''\},\ldots,\{\ell_{d-m}''\}\Big\}.\]

Finally, for all $\mathbf{j},\mathbf{k} \in \Z^d$, we set
\begin{equation}\label{eq:tensor-psi}
\psi_{\mathbf{j},\mathbf{k}}:=\bigotimes_{\ell=1}^{d}  \psi_{{j}_\ell,{k}_\ell},
\end{equation}
and we consider the chaotic random variable
\begin{equation}\label{eqn:rv:eps}
 \varepsilon_{\textbf{j},\textbf{k}} := \int_{\R^d}' \psi_{\mathbf{j},\mathbf{k}}(x_1, \ldots,x_d) \, dB(x_1)\ldots dB(x_d).
\end{equation}

Theorem 2.8 and Theorem 2.12 of \cite{ayachehamonierloosveldt} can then be summarized together as follows.

\begin{Thm}\label{thm:mainfirstpaper}
The random series
\begin{equation}\label{eq:approximprocess}
X_{\mathbf{h},J}^{(d)}(t)= 2^{-J(h_1+\ldots+h_d-d)} \sum_{\mathbf{k} \in \Z^d} \sigma_{J,\textbf{k}}^{(\mathbf{h})} \int_0^t \prod_{\ell=1}^d \Phi_\Delta^{(h_\ell-1/2)}(2^J s-k_\ell) \, ds,
\end{equation}
is almost surely, for each $J \in \N$, uniformly convergent in $t$ on each compact interval $I \subset \R_+$. Moreover, for all such $I$, there exists an almost surely finite random variable (depending on $I$) for which one has, almost surely, for each $J \in \N$,
\begin{align}
\| X^{(d)}_{\mathbf{h}} &-X^{(d)}_{\mathbf{h},J} \|_{I,\infty} \nonumber \\ & = \left \| \sum_{\substack{(\mathbf{j},\mathbf{k}) \in (\Z^d)^2 \\  \displaystyle\max_{\ell \in [\![1,d]\!]} j_\ell \geq J }} 2^{ j_1 (1-h_1)+\cdots + j_d (1-h_d)} \varepsilon_{\textbf{j},\textbf{k}}  \int_0^t \prod_{\ell=1}^d \psi_{h_\ell}(2^{j_\ell}s-k_\ell) \, ds \right \|_{I,\infty} \label{eq:theserieinrateofconvergence}\\
& \leq C J^{\frac{d}{2}} 2^{-J(  h_1+ \cdots + h_d -d + 1/2)}.\label{eq:rateofconvergence}
\end{align}
\end{Thm}

In view of Theorem \ref{thm:mainfirstpaper}, we call $\{X_{\mathbf{h},J}^{(d)}(t) \}_{t \in \R_+}$ the \textit{approximation process at scale $J$}. Unfortunately, it has not yet a form which is well adapted to numerical simulations. In order to overcome this difficulty we need to replace it by the stochastic process $\{S_{\mathbf{h},J}^{(d)}(t) \}_{t \in \R_+}$, called the \textit{simulation process at scale $J$}, which will be defined in the sequel.
%it is not computable, because of the series involved in the expression \eqref{eq:approximprocess}. 
%For numerical simulation, we need to replace it by a \textit{simulation process} (of scale $J$) $\{S_{\mathbf{h},J}^{(d)}(t) \}_{t \in \R_+}$ which is more appropriate. 
We aim at doing so without altering the rate of convergence given in the second part of Theorem \ref{thm:mainfirstpaper}. On this purpose, we recall the following definition of some important sets.

\begin{Def}\cite[Definition 4.3]{ayachehamonierloosveldt}
Let $a$ be a fixed real number satisfying $1/2 < a < 1$. For all $(j,k)\in \Z_+ \times \Z$, let us denote by $B_{j,k}$ the interval 
\begin{equation}\label{def:bjk}
B_{j,k}:= [ k2^{-j}-2^{-ja},k2^{-j}+2^{-ja} ].
\end{equation}
For all $j\in\N$ and for all $t\in\R_+$, we consider the three disjoint subsets of $\Z$:
\begin{align}
D_j^1(t) & := \lbrace k\in\Z : B_{j,k} \subseteq [0,t] \rbrace,  \label{def:dj1}\\
D_j^2(t) & := \lbrace k\in\Z\setminus D_j^1(t) : B_{j,k} \cap [0,t] \neq \emptyset \rbrace,\label{def:dj2}\\
D_j^3(t) & := \lbrace k\in\Z  : B_{j,k} \cap [0,t] = \emptyset \rbrace. \label{def:dj3}
\end{align}
These three sets, depending on $t$ and $a$, form a partition of $\Z$:
\[
\Z = \bigcup_{\ell=1}^3 D_j^{\ell}(t).
\]
\end{Def}
The set $(D_j^1(t))^d$ is the cartesian product of the set $D_j^1(t)$ $d$-times with itself. In \cite{ayachehamonierloosveldt}, we somehow remarked that the rate of convergence \eqref{eq:rateofconvergence} is mainly determined by the terms of the series in \eqref{eq:theserieinrateofconvergence} whose indices $\mathbf{k}$ belong to some $(D_j^1(t))^d$ sets, with $j \geq J$ and $t \in I$. Therefore, the idea behind the definition of $S_{\mathbf{h},J}^{(d)}(t)$, with $J \in \N$ and $t \geq 0$, is to use an appropriate enlargement of the ``diagonal set'' of $(D_J^1(t))^d$.

\begin{Def}
Let $\varepsilon$ be a fixed positive real number. For all $J\in\N$ and for all $t\in\R_+$, we consider the three subsets
\begin{align}
\mathcal{J}_J^1(t) & := \lbrace \mathbf{k} \in (D_J^1(t))^d \, : \, \max_{ \ell, \ell' \in  [\![1,d]\!]} |k_\ell - k_{\ell'}| \leq 2^{\varepsilon J} \rbrace,  \label{def:JJ1}\\
\mathcal{J}_J^2(t) & := \lbrace \mathbf{k} \in \Z^d \, : \,  \exists \, n \text{ such that } k_n \in D_J^2(t)\rbrace,\label{def:JJ2}\\
\mathcal{J}_J^3(t) & := \lbrace \mathbf{k} \in \Z^d \, : \,  \exists \, n \text{ such that } k_n \in D_J^3(t) \}. \label{def:JJ3}
\end{align}
\end{Def}

The simulation process $\{S_{\mathbf{h},J}^{(d)}(t) \}_{t \in \R_+}$ is then defined as follows.

\begin{Def}
For all $J \in \N$, the \textit{simulation process at scale $J$} of the generalized Hermite process $\{X_\mathbf{h}^{(d)}(t)\}_{t \in \R_+}$ is the process defined, for all $t \in \R_+$, by
\begin{equation}\label{eq:simulationprocess}
S_{\mathbf{h},J}^{(d)}(t) = 2^{-J(h_1+\ldots+h_d-d+1)} \sum_{\mathbf{k} \in \mathcal{J}_J^1(t)} \sigma_{J,\textbf{k}}^{(\mathbf{h})} \int_{\R} \prod_{\ell=1}^d \Phi_\Delta^{(h_\ell-1/2)}(s-k_\ell) \, ds.
\end{equation}
\end{Def}

\begin{Rmk}
From consecutive uses of the Parseval-Plancherel identity and the inclusion \eqref{eqn:supportfractscalfunction}, one can rewrite the integrals appearing in the right-hand side of \eqref{eq:simulationprocess} as integrals of compactly supported functions, see equation \eqref{prodscalfuncenfourier} below. This fact is particularly interesting for numerical computations.
\end{Rmk}

The first main result of this paper is that the sequence of simulation process $(\{S_{\mathbf{h},J}^{(d)}(t)\}_{t \in \R_+})_{J \in \N}$ converges, almost surely,  uniformly on any compact set of $\R_+$ to the generalized Hermite process $\{X_\mathbf{h}^{(d)}(t)\}_{t \in \R_+}$ with the same rate of convergence as the one established for the approximation process, in Theorem \ref{thm:mainfirstpaper}.

\begin{Thm}\label{thm:main}
For any compact interval $I \subset \R_+$, there exists an almost surely finite random variable $C$ (depending on $I$) for which one has, almost surely, for each $J \in \N$,
\begin{equation}\label{eqn:thm:main1}
\| X^{(d)}_{\mathbf{h}}-S_{\mathbf{h},J}^{(d)} \|_{I,\infty}  \leq C J^{\frac{d}{2}} 2^{-J(  h_1+ \cdots + h_d -d + 1/2)}.
\end{equation}
\end{Thm}

To state our second result, we need to look carefully at the definition of the set (\ref{def:dj1}). Indeed, for each fixed $J \in \N$ and $t\in\R_+$, we denote by $m_{J,t}$ the integer part of the real number  $2^Jt-2^{J(1-a)}$, that is $m_{J,t}:=\lfloor 2^Jt-2^{J(1-a)} \rfloor$. In view of the definition \eqref{def:dj1} of the set $D_J^1(t)$, it turns out that
\begin{equation}
\label{rem:finiD1bis}
D_J^1(t)=\left\{ \begin{array}{cl}
                       \emptyset & \text{if } t \in [0, 2^{1-Ja})\\
                       D_j^1(m_{J,t}2^{-J}+2^{-Ja}) & \text{if } t \in [2^{1-Ja},\infty).
                     \end{array} \right.
\end{equation}
Then, we can derive from \eqref{rem:finiD1bis} that the process $\{S_{\mathbf{h},J}^{(d)} (t)\}_{t \in \R^+}$ is a piecewise constant random function, that is a random walk. More precisely, if $\mathcal{I}_J$ stands for the set
\[
\mathcal{I}_J := \N \cap (2 ^{J(1-a)}-1,+ \infty),
\]
one has
\begin{equation}
\label{eq:add1}
S_{\mathbf{h},J}^{(d)}(t) = \sum_{m \in \mathcal{I}_J}  s_{m,J}\mathbbm{1}_{\lambda_{m,J}^{(a)}}(t), \quad \mbox{for every $t\in\R_+$,}
\end{equation}
where, for all $m \in \mathcal{I}_J$, we set the random variable $s_{m,J} := S_{\mathbf{h},J}^{(d)}(m 2^{-J}+2^{-aJ})$ and the deterministic  bounded interval $\lambda_{m,J}^{(a)}:=[m 2^{-J}+2^{-aJ},(m+1) 2^{-J}+2^{-aJ})$.

In virtue of Theorem \ref{thm:main}, when $J$ is large enough, the random walk in \eqref{eq:add1} can be used for numerically simulating the generalized Hermite process $\{X^{(d)}_{\mathbf{h}}(t)\}_{t \in \R_+}$. Nevertheless, from Remark \ref{rmk:holder}, we know that sample paths of the latter process are almost surely continuous functions. Thus, one could prefer to simulate them through a numerically computable stochastic process having continuous sample paths. This can be simply done by taking a linear interpolation between the random points $(m 2^{-J}+2^{-aJ},s_{m,J})$. Namely, we consider the process $\{\widetilde{S}_{\mathbf{h},J}^{(d)}(t)\}_{t \in \R_+}$ defined, for all $t \in \R_+$, by
\begin{align}\label{eqn:interpolin}
&\widetilde{S}_{\mathbf{h},J}^{(d)}(t):=\frac{s_{m_0,J}}{|\lambda_{0,J}^{(a)}|} t \mathbbm{1}_{\lambda_{0,J}^{(a)}}(t)\\
&\hspace{1cm}+\sum_{m \in \mathcal{I}_J} \left(2^J\left(s_{m+1,J}-s_{m,J}\right)(t-(m 2^{-J}+2^{-aJ}))+ s_{m,J} \right) \mathbbm{1}_{\lambda_{m,J}}(t),\nonumber
\end{align}
with $m_0:= \inf \mathcal{I}_J$, $\lambda_{0,J}^{(a)}:= [0,m_02^{-J}+2^{-aJ})$ and $|\lambda_{m_0,J}^{(a)}|=m_02^{-J}+2^{-aJ}$.

The following theorem shows that the sequence of simulation processes $\big (\{\widetilde{S}_{\mathbf{h},J}^{(d)}(t)\}_{t \in \R_+}\big)_{J \in \N}$ almost surely converges uniformly on any compact interval of $\R_+$ to $\{X^{(d)}_{\mathbf{h}}(t)\}_{t \in \R_+}$, with the same rate of convergence as the one in Theorem \ref{thm:main}. 
%At this stage, we recall the notation introduced in Remark \ref{rmk:holder}: $H \in (1/2,1)$ is the real number given by
%\[ H = \sum_{\ell=1}^d h_\ell-d+1. \]
%
%Therefore, we show in our second main result that, under some condition of the real number $a$, introduced in (\ref{def:bjk}), we can replace the simulation process  $S_{\mathbf{h},J}^{(d)}$ in \ref{thm:main} by the interpolate process $\widetilde{S}_{\mathbf{h},J}^{(d)}$. Namely, one gets:

\begin{Thm}\label{thm:main2}
If $a > 1-\frac{1}{2H}$, where $H\in (1/2,1)$ is as in \eqref{eq:add2}. Then, for any compact interval $I \subset \R_+$, there exists an almost surely finite random variable $C'$ (depending on $I$) for which one has, almost surely, for each $J \in \N$,
\[
\| X^{(d)}_{\mathbf{h}}-\widetilde{S}_{\mathbf{h},J}^{(d)} \|_{I,\infty}  \leq C' J^{\frac{d}{2}} 2^{-J(  h_1+ \cdots + h_d -d + 1/2)}.
\]
\end{Thm}

\section{Proofs of the main results} \label{sec:proofs}

\subsection{Proof of Theorem \ref{thm:main}}

Our strategy to prove Theorem \ref{thm:main} consists in writing, for all $J \in \N$, the approximation error $X^{(d)}_{\mathbf{h}}-S_{\mathbf{h},J}^{(d)}$  as the sum of random variables, indexed by the sets $(D_J^1(t))^d$, $\mathcal{J}_J^1(t)$, $\mathcal{J}_J^2(t)$ and $\mathcal{J}_J^3(t)$ and to bound the sup-norm of these random variables in a suitable way. On this purpose, let us first remark that \cite[Proposition 2.7]{ayachehamonierloosveldt} as well as \cite[Lemma 2.2.17]{theseyassine} allow to affirm the existence of $C^*$ a positive finite random variable and $\Omega^*$ an event of probability $1$, such that, on $\Omega^*$, for all $J \in \N$ and $\mathbf{k} \in \Z^d$, one has
\begin{equation}\label{eqn:boundonsigma}
|\sigma_{J,\mathbf{k}}^{(\mathbf{h})} | \leq C^* \prod_{\ell =1}^d \sqrt{\log(3+|J|+|k_\ell|)}.
\end{equation}

\begin{Lemma}\label{lemmeJ2}
Let $T >2$ be a fixed real number. There exits a positive almost surely finite random variables $C$ such that, for all $J \in \N$, on $\Omega^*$, the random variable
\begin{align*}
\mathcal{G}^2_{J} &:=  \sup_{t \in [0,T]} \left \{ 2^{-J(h_1+\ldots+h_d-d)} \sum_{\mathbf{k} \in \mathcal{J}_J^2(t)} |\sigma_{J,\textbf{k}}^{(\mathbf{h})}| \left|\int_0^t \prod_{\ell=1}^d \Phi_\Delta^{(h_\ell-1/2)}(2^J s-k_\ell) \, ds \right| \right\}
\end{align*}
is bounded from above by
$ C J^{\frac{d}{2}} \left(\log(3+J)\right)^{\frac{d-1}{2}} 2^{-J(h_1+ \cdots + h_d +a -d)}.$
\end{Lemma}

\begin{proof}
Let $L>1$ be a fixed real number, for all $t \in [0,T]$, using the definition \eqref{def:JJ2}, the inequality \eqref{eqn:boundonsigma} and the fast decay property \eqref{maj:psih}, we have, on $\Omega^*$,
\begin{align*}
&2^{-J(h_1+\ldots+h_d-d)} \sum_{\mathbf{k} \in \mathcal{J}_J^2(t)} |\sigma_{J,\textbf{k}}^{(\mathbf{h})}| \left|\int_0^t \prod_{\ell=1}^d \Phi_\Delta^{(h_\ell-1/2)}(2^J s-k_\ell) \, ds \right| \\
& \quad \leq C_1\sup_{ n \in [\![1,d]\!]}  \left(2^{-J(h_1+\ldots+h_d-d)} \sum_{k_n \in D_J^2(t)} \sum_{\substack{k_{\ell} \in \Z \\  \ell \neq n} } \int_0^T \prod_{\ell=1}^d \frac{\sqrt{\log(3+J+|k_\ell|)}}{(3+|2^J s-k_\ell|)^L} \, ds \right),
\end{align*}
where $C_1$ is a positive finite random variable, not depending on $t$ or $J$. We proceed as in the proof of \cite[Lemma 4.6]{ayachehamonierloosveldt} to get, on $\Omega^*$,
\begin{align*}
| \mathcal{G}^2_{J} | & \leq C_2 2^{-J(h_1+ \cdots + h_d +a -d)} \sqrt{1+J} \log(3+J+2^J T)^\frac{d-1}{2} \\
& \leq C_3  2^{-J(h_1+ \cdots + h_d +a -d)} J^{\frac{d}{2}} \log(3+J)^\frac{d-1}{2},
\end{align*}
where $C_2$ and $C_3$ are positive finite random variables, not depending on $J$.
\end{proof}

\begin{Lemma}\label{lemmeJ3}
Let $T >2$ and $L \geq 2^{-1}(1-a)^{-1}+1$ be two fixed real numbers. There exits a positive almost surely finite random variables $C$ such that, for all $J \in \N$, on $\Omega^*$, the random variable
\begin{align*}
\mathcal{G}^3_{J} &:=  \sup_{t \in [0,T]} \left \{ 2^{-J(h_1+\ldots+h_d-d)} \sum_{\mathbf{k} \in \mathcal{J}_J^3(t)} |\sigma_{J,\textbf{k}}^{(\mathbf{h})}| \left|\int_0^t \prod_{\ell=1}^d \Phi_\Delta^{(h_\ell-1/2)}(2^J s-k_\ell) \, ds \right| \right\}
\end{align*}
is bounded from above by
$ C J^{\frac{d+1}{2}} \left(\log(3+J)\right)^{\frac{d-1}{2}} 2^{-J(h_1+ \cdots + h_d +(L-1)(1-a) -d)}.$
\end{Lemma}

\begin{proof}
For all $t \in [0,T]$, using the definition \eqref{def:JJ2}, the inequality \eqref{eqn:boundonsigma} and the fast decay property \eqref{maj:psih}, we have, on $\Omega^*$,
\begin{align*}
& 2^{-J(h_1+\ldots+h_d-d)} \sum_{\mathbf{k} \in \mathcal{J}_J^3(t)} |\sigma_{J,\textbf{k}}^{(\mathbf{h})}| \left|\int_0^t \prod_{\ell=1}^d \Phi_\Delta^{(h_\ell-1/2)}(2^J s-k_\ell) \, ds \right| \\
& \quad \leq C_1\sup_{ n \in [\![1,d]\!]}  \left(2^{-J(h_1+\ldots+h_d-d)} \sum_{k_n \in D_J^3(t)} \sum_{\substack{k_{\ell} \in \Z \\  \ell \neq n} } \int_0^T \prod_{\ell=1}^d \frac{\sqrt{\log(3+J+|k_\ell|)}}{(3+|2^J s-k_\ell|)^L} \, ds \right),
\end{align*}
where $C_1$ is a positive finite random variable, not depending on $t$ or $J$. We proceed as in the proof of \cite[Lemma 4.5]{ayachehamonierloosveldt} to get, on $\Omega^*$,
\begin{align*}
| \mathcal{G}^3_{J} | & \leq C_2 (1+J) \log(3+J+2^J T)^\frac{d-1}{2}  2^{-J(h_1+ \cdots + h_d +(L-1)(1-a) -d)} \\
& \leq C_3  J^{\frac{d+1}{2}}  \log(3+J)^\frac{d-1}{2}  2^{-J(h_1+ \cdots + h_d +(L-1)(1-a) -d)},
\end{align*}
where $C_2$ and $C_3$ are positive finite random variables, not depending on $J$.
\end{proof}

\begin{Lemma}\label{lemmaint}
Let $T >2$ and $L>2$ be two fixed real numbers. There exits a positive almost surely finite random variable $C$ such that, for all $J \in \N$, on $\Omega^*$, the random variable
\begin{align*}
\mathcal{G}^1_{J} &:=  \sup_{t \in [0,T]} \left \{ 2^{-J(h_1+\ldots+h_d-d)} \sum_{\mathbf{k} \in (D_J^1(t))^d} |\sigma_{J,\textbf{k}}^{(\mathbf{h})}| \left|\int_{\R \setminus [0,t]} \prod_{\ell=1}^d \Phi_\Delta^{(h_\ell-1/2)}(2^J s-k_\ell) \, ds \right| \right\}
\end{align*}
is bounded from above by $C J^{\frac{d}{2}}\, 2^{-J(d(L-1)(1-a)+h_1+ \cdots + h_d  -d+1)}$.
\end{Lemma}

\begin{proof}
For all $t \in [0,T]$, using the inequality \eqref{eqn:boundonsigma}, the fast decay property \eqref{maj:psih}, the inequality $|k| \leq 2^J T$, for all $k \in D_J^1(t)$, the inequality $2^JT \geq J$, \cite[Lemma B.2]{ayachehamonierloosveldt} and the definition \eqref{def:dj1}, we have, on $\Omega^*$,
\begin{align*}
 & 2^{-J(h_1+\ldots+h_d-d)} \sum_{\mathbf{k} \in (D_J^1(t))^d} |\sigma_{J,\textbf{k}}^{(\mathbf{h})}| \left|\int_{\R \setminus [0,t]} \prod_{\ell=1}^d \Phi_\Delta^{(h_\ell-1/2)}(2^J s-k_\ell) \, ds \right|  \\
 & \leq \quad C_1 2^{-J(h_1+\ldots+h_d-d)} \int_{\R \setminus [0,t]} \prod_{\ell=1}^d \left(\sum_{k_\ell \in D_J^1(t)}\frac{\sqrt{\log(3+J+|k_\ell|)}}{(3+|2^Js-k_\ell|)^L} \right) \, ds \\
  & \leq \quad  C_1 2^{-J(h_1+\ldots+h_d-d)} \int_{\R \setminus [0,t]} \prod_{\ell=1}^d \left(\sum_{k_\ell \in D_J^1(t)}\frac{\sqrt{\log(3+2^{J+1}T)}}{(3+|2^Js-k_\ell|)^L} \right) \, ds \\
 & \leq \quad C_2 J^\frac{d}{2} 2^{-J(h_1+\ldots+h_d-d)} \int_{\R \setminus [0,t]} \prod_{\ell=1}^d \left(\sum_{k_\ell \in D_J^1(t)}\frac{1}{(3+|2^Js-k_\ell|)^L} \right) \, ds \\
 & \leq \quad C_2 J^\frac{d}{2} 2^{-J(h_1+\ldots+h_d-d)}  (I_J^1(t) + I_J^2(t)),
\end{align*}
where $C_1$ and $C_2$ are positive finite random variable, not depending on $t$ or $J$, and we have set
\[ I_J^1(t) := \int_{t}^{+ \infty} \prod_{\ell=1}^d \left( \sum_{k_\ell \leq 2^J t-2^{J(1-a)}} \frac{1}{(3+2^Js-k_\ell)^L} \right) \, ds \]
and
\[ I_J^2(t) := \int_{-\infty}^{0} \prod_{\ell=1}^d \left( \sum_{k_\ell \geq 2^{J(1-a)}} \frac{1}{(3+k_\ell-2^J s)^L} \right) \, ds .\]

To bound $I_J^1(t)$, we use the change of variable $y=2^J(s-t)$ and the fact that the function $y \mapsto (2+y)^{-L}$ is decreasing on $\R_+$ to get
\begin{align*}
I_J^1(t) &= 2^{-J} \int_0^{+ \infty}  \prod_{\ell=1}^d \left( \sum_{k_\ell \leq 2^J t-2^{J(1-a)}} \frac{1}{(3y++2^Jt-k_\ell)^L} \right) \, dy \\
& \leq 2^{-J} \int_0^{+ \infty}  \prod_{\ell=1}^d \left( \int_0^{+ \infty} \frac{dz}{(2+y+2^{J(1-a)}+z)^L} \right) \, dy \\
& =2^{-J} \int_0^{+ \infty} \frac{(L-1)^d}{(2+y+2^{J(1-a)})^{d(L-1)}} \, dy \\
&= 2^{-J} (L-1)^d (d(L-1)-1) (2+2^{J(1-a)})^{-(d(L-1)-1)}.
\end{align*}
We bound $I_J^2(t)$ in the same way and get the conclusion.
\end{proof}

To get the proof of Theorem \ref{thm:main} done, it only remains us to bound the random variables
\[ 2^{-J(h_1+\ldots+h_d-d+1)} \sum_{\mathbf{k} \in (D_J^1(t))^d \setminus \mathcal{J}_J^1(t)} \sigma_{J,\textbf{k}}^{(\mathbf{h})} \int_{\R} \prod_{\ell=1}^d \Phi_\Delta^{(h_\ell-1/2)}(s-k_\ell) \, ds.\]
Let us then consider $\mathbf{k} \in (D_J^1(t))^d \setminus \mathcal{J}_J^1(t)$, up to a permutation of the indices, one can assume $|k_1-k_2| > 2^{\varepsilon J}$. Using the change of variable $s=x+k_2$, we write
\[ \int_{\R} \prod_{\ell=1}^d \Phi_\Delta^{(h_\ell-1/2)}(s-k_\ell) \, ds = \int_{\R} \Phi_\Delta^{(h_1-1/2)}(x+k_2-k_1)\Phi_\Delta^{(h_2-1/2)}(x) \prod_{\ell=3}^d \Phi_\Delta^{(h_\ell-1/2)}(x+k_2-k_\ell) \, dx. \]
We use the Parseval-Plancherel identity and the equality $\widehat{f g}= \widehat{f} \star \widehat{g}$, for all $f,g \in L^2$, to get
\begin{align}
\int_{\R} & \Phi_\Delta^{(h_1-1/2)}(x+k_2-k_1)\Phi_\Delta^{(h_2-1/2)}(x) \prod_{\ell=3}^d \Phi_\Delta^{(h_\ell-1/2)}(x+k_2-k_\ell) \, dx \nonumber\\ & = (2\pi)^{-d}\int_{\R} e^{i(k_2-k_1) \xi} \reallywidehat{\Phi_\Delta^{(h_1-1/2)}}(\xi) \overline{\left(\reallywidehat{\Phi_\Delta^{(h_2-1/2)} \prod_{\ell=3}^d \Phi_\Delta^{(h_\ell-1/2)}(\cdot+k_2-k_\ell) }\right)(\xi)} \, d\xi  \nonumber \\
&= (2\pi)^{-d}  \int_{\R} e^{i(k_2-k_1) \xi} \reallywidehat{\Phi_\Delta^{(h_1-1/2)}}(\xi) \nonumber \\
&\hspace{2cm}\times\overline{\left(\reallywidehat{\Phi_\Delta^{(h_2-1/2)}} \star  e^{i(k_2-k_3) \cdot}\reallywidehat{\Phi_\Delta^{(h_3-1/2)}} \star \cdots \star e^{i(k_2-k_d)\cdot}\reallywidehat{\Phi_\Delta^{(h_d-1/2)}}\right)(\xi)} \, d\xi \nonumber\\
& = \int_{\R^{d-1}} e^{i(k_2-k_1) \xi} e^{i(k_3-k_2) \eta_3} \cdots e^{i(k_d-k_2) \eta_d} F (\xi,\eta_3,\dots,\eta_d) \, d\xi d\eta_3 \cdots d\eta_d, \label{prodscalfuncenfourier}
\end{align}
where we put
\[F \, : \, (\xi,\eta_3,\dots,\eta_d) \mapsto (2\pi)^{-d}\,\reallywidehat{\Phi_\Delta^{(h_1-1/2)}}(\xi) \overline{\reallywidehat{\Phi_\Delta^{(h_2-1/2)}}}(\xi-(\eta_3+\cdots+\eta_d))  \prod_{\ell=3}^d \overline{\reallywidehat{\Phi_\Delta^{(h_\ell-1/2)}}}(\eta_\ell).  \]
But, as $F \in \mathcal{S}(\R^{d-1})$, we conclude that its Fourier transform
\[ \mathcal{F} \, : \, \mathbf{x} \mapsto \int_{\R^{d-1}} e^{i \langle (x_1,\dots,x_{d-1}),(\xi,\eta_3,\dots,\eta_d) \rangle}  F (\xi,\eta_3,\dots,\eta_d) \, d\xi d\eta_3 \cdots d\eta_d \]
also belongs to the space $\mathcal{S}(\R^{d-1})$. Therefore, for all $L>0$, there exists a deterministic constant $c_L>0$ such that
\[ \sup_{\mathbf{x} \in \R^{d-1}} (1+|x_1|)^L \dots (1+|x_{d-1}|)^L |\mathcal{F}(\mathbf{x})| < c_L.\]
In particular, as we assume $|k_1-k_2| > 2^{\varepsilon J}$, we get
\begin{align}
\left| \int_{\R} \prod_{\ell=1}^d \Phi_\Delta^{(h_\ell-1/2)}(s-k_\ell) \, ds \right| & \leq c_L \prod_{\ell=1, \ell \neq 2}^d \frac{1}{(1+|k_\ell-k_2|)^{L+1}} \nonumber \\
& \leq c_L 2^{- \varepsilon J L} \prod_{\ell=1, \ell \neq 2}^d \frac{1}{(1+|k_\ell-k_2|)^{L}}. \label{eqn:clef}
\end{align}

\begin{Lemma}\label{lemm:final}
Let $T >2$ and $L>\max\{(2 \varepsilon)^{-1},1\}$ be two fixed real numbers. There exits a positive almost surely finite random variable $C$ such that, for all $J \in \N$, on $\Omega^*$, the random variable
\[ \mathcal{G}_J:= \sup_{t \in [0,T]} \left\{ 2^{-J(h_1+\ldots+h_d-d+1)} \sum_{\mathbf{k} \in (D_J^1(t))^d \setminus \mathcal{J}_J^1(t)} |\sigma_{J,\textbf{k}}^{(\mathbf{h})}| \left| \int_{\R} \prod_{\ell=1}^d \Phi_\Delta^{(h_\ell-1/2)}(s-k_\ell) \, ds \right| \right\}\]
is bounded by $C J^{\frac{d}{2}} 2^{-J(h_1+\ldots+h_d-d +\varepsilon L)}$
\end{Lemma}
\begin{proof}
For all $t \in [0,T]$, using the inequality \eqref{eqn:boundonsigma}, the inequality \eqref{eqn:clef} and the fact that for all $k \in D_J(t)$, $|k| \leq 2^J T$, we get
\begin{align*}
 & 2^{-J(h_1+\ldots+h_d-d+1)} \sum_{\mathbf{k} \in (D_J^1(t))^d \setminus \mathcal{J}_J^1(t)} |\sigma_{J,\textbf{k}}^{(\mathbf{h})}| \left| \int_{\R} \prod_{\ell=1}^d \Phi_\Delta^{(h_\ell-1/2)}(s-k_\ell) \, ds \right| \\
& \quad\leq C_1  2^{-J(h_1+\ldots+h_d-d+1)} (\log(3+J+2^JT))^\frac{d}{2} 2^{- \varepsilon J L} (\#(D_J^1(t)) \sup_{x \in \R} \left( \sum_{k \in \Z} \frac{1}{(1+|k-x|)^L} \right)^{d-1} \\
& \quad \leq C_2 J^{\frac{d}{2}} 2^{-J(h_1+\ldots+h_d-d +\varepsilon L)},
\end{align*}
where $C_1$ and $C_2$ are positive finite random variable, not depending on $t$ or $J$.
\end{proof}

The proof of Theorem \ref{thm:main} is then a direct consequence of the Lemmata in this section and Theorem \ref{thm:mainfirstpaper}.

\begin{proof}[Proof of Theorem \ref{thm:main}]
Without loss of generality, one can assume that the compact interval $I$ in the statement of the theorem is of the form $I=[0,T]$ for a fixed real number $T>2$. Of course, for all $J \in \N$, we have
\begin{align*}
\| X^{(d)}_{\mathbf{h}}-S_{\mathbf{h},J}^{(d)} \|_{I,\infty} & \leq \| X^{(d)}_{\mathbf{h}}-X_{\mathbf{h},J}^{(d)} \|_{I,\infty} + \| X_{\mathbf{h},J}^{(d)} -S_{\mathbf{h},J}^{(d)} \|_{I,\infty} \\
& \leq  \| X^{(d)}_{\mathbf{h}}-X_{\mathbf{h},J}^{(d)} \|_{I,\infty}  + \mathcal{G}^1_{J}+ \mathcal{G}^2_{J} + \mathcal{G}^3_{J} + \mathcal{G}_J
\end{align*}
and we conclude using Theorem \ref{thm:mainfirstpaper} and Lemmata \ref{lemmeJ2}, \ref{lemmeJ3}, \ref{lemmaint} and \ref{lemm:final}.
\end{proof}

\subsection{Proof of Theorem \ref{thm:main2}}

We have enough materials to directly prove Theorem \ref{thm:main2}.

\begin{proof}[Proof of Theorem \ref{thm:main2}]
For the sake of simpleness in notation, let us assume $I=[0,T]$, for some $T>0$. For all $J \in \N$, we write
\[ \mathcal{I}_J(T) := \N \cap (2 ^{J(1-a)}-1,2^J T-2^{J(1-a)}].\]
For all such a $J$, we get, using the triangle inequality,
\begin{align*}
 \| {S}_{\mathbf{h},J}^{(d)}-\widetilde{S}_{\mathbf{h},J}^{(d)} \|_{I,\infty} & \leq \max \left \{ |s_{m_0,J}| ,\max_{m \in \mathcal{I}_J(T)} \left |s_{m+1,J}-s_{m,J} \right | \right \} \\
& \leq 2 \| X^{(d)}_{\mathbf{h}}-S_{\mathbf{h},J}^{(d)} \|_{I,\infty}  +  \max \left \{ A_J , B_J\right \} 
\end{align*}
where we set
\[A_J := |X^{(d)}_{\mathbf{h}}(m_02^{-J}+2^{-aJ})-X^{(d)}_{\mathbf{h}}(0)| \]
and
\[B_J := \max_{m \in \mathcal{I}_J(T)} \left |X^{(d)}_{\mathbf{h}}((m+1)2^{-J}+2^{-aJ})-X^{(d)}_{\mathbf{h}}(m2^{-J}+2^{-aJ}) \right |. \]
Let us consider $H-\frac{1}{2} < \gamma < H$ such that $a \gamma> H-\frac{1}{2}$. For all $m \in \mathcal{I}_J(T)$, we use the inequality \eqref{eqn:holderrmk} to find a positive finite random variable $C_{T,\gamma}$ such that, almost surely,
\[ \left |X^{(d)}_{\mathbf{h}}((m+1)2^{-J}+2^{-aJ})-X^{(d)}_{\mathbf{h}}(m2^{-J}+2^{-aJ}) \right | \leq C_{T,\gamma} 2^{-\gamma J}  \]
and thus, as $\gamma > H-\frac{1}{2}=h_1+ \cdots + h_d -d + 1/2$, we deduce
\begin{equation}\label{eqn:conditisura}
B_J \leq C_{T,\gamma} J^{\frac{d}{2}} 2^{-J(  h_1+ \cdots + h_d -d + 1/2)}.
\end{equation}

We also use the inequality\eqref{eqn:holderrmk} to get, as $0 < m_0 \leq 2^{J(1-a)}$,
\begin{align*}
 |X^{(d)}_{\mathbf{h}}(m_02^{-J}+2^{-aJ})-X^{(d)}_{\mathbf{h}}(0)| & \leq C_\gamma |m_02^{-J}+2^{-aJ}|^\gamma \\
 & \leq 2^\gamma C_{T,\gamma} 2^{-a \gamma J}.
\end{align*}
Then, as we choose $\gamma$ such that $a \gamma> H-\frac{1}{2}$, we get
\[B_J \leq  2^\gamma C_{T,\gamma} J^{\frac{d}{2}} 2^{-J(  h_1+ \cdots + h_d -d + 1/2)}.\]
Thus, we obtain, using the triangle inequality,
\[\| X^{(d)}_{\mathbf{h}}-\widetilde{S}_{\mathbf{h},J}^{(d)} \|_{I,\infty} \leq 3 \| X^{(d)}_{\mathbf{h}}-S_{\mathbf{h},J}^{(d)} \|_{I,\infty} + 4 C_{T,\gamma} J^{\frac{d}{2}} 2^{-J(  h_1+ \cdots + h_d -d + 1/2)}  \]
and the conclusion follows from inequality \eqref{eqn:thm:main1}.
\end{proof}

\begin{Rmk}
In the last Theorem, the condition $a > 1-\frac{1}{2H}$ is only required to obtain the inequality \eqref{eqn:conditisura}. In particular, if $\frac{1}{2}<a<1$ and $I$ is a compact interval included in $[m_02^{-J}+2^{-aJ},+\infty)$, one can find an almost surely finite random variable $C''$ (depending on $I$ but not $J$) for which
\[
\| X^{(d)}_{\mathbf{h}}-\widetilde{S}_{\mathbf{h},J}^{(d)} \|_{I,\infty}  \leq C' J^{\frac{d}{2}} 2^{-J(  h_1+ \cdots + h_d -d + 1/2)},
\]
without any additional assumptions on the value of the parameter $a$.
\end{Rmk}

\section{Examples} \label{sec:examples}

We use the process defined in \eqref{eqn:interpolin} to numerically simulate the Rosenblatt process and the Hermite process of order 3. We refer to \cite[Chap. 4]{Pipiras_Taqqu_2017} for more information on these processes.

\subsection{Simulation of the Rosenblatt process}

The Rosenblatt process of Hurst parameter $H \in (1/2,1)$ is the process introduced in \eqref{eqn:def:ghp} with $d=2$ and $\mathbf{h}=(1+\frac{H-1}{2},1+\frac{H-1}{2})$. To generate the random variables $(\sigma_{J,\textbf{k}}^{(\mathbf{h})})_{J \in \N, \mathbf{k} \in \Z^d}$, one can use the more explicit equation, derived from \eqref{eqn:defofsigma},
\begin{equation}\label{eqn:rosenblatt1}
\sigma_{J,k_1,k_2}^{(1+\frac{H-1}{2},1+\frac{H-1}{2})}=Z_{J,k_{1}}^{(\frac{H}{2})}Z_{J,k_{2}}^{(\frac{H}{2})}-\mathbb{E}[Z_{J,k_{1}}^{(\frac{H}{2})}Z_{J,k_{2}}^{(\frac{H}{2})}].
\end{equation}
where $(Z_{J,\ell}^{(\frac{H}{2})})_{\ell \in \N}$ is the Gaussian FARIMA $(0,\frac{H}{2},0)$ sequence associated to the sequence of Gaussian random variables $(2^{\frac{J}{2}} \int_{\R} \phi(2^J x-k) dB(x))_{k\in\Z}$. Note that we have \cite[Remark 3.5]{ayachehamonierloosveldt}, for all $J \in \N$ and $\mathbf{k} \in \Z^2$
\begin{equation}\label{explicitexpressioncovrosenb}
\mathbb{E}[Z_{J,k_{1}}^{(\frac{H}{2})}Z_{J,k_{2}}^{(\frac{H}{2})}] = \frac{1}{2\pi} \int_0^{2 \pi} e^{i \xi (k_2-k_1)} \left| 2 \sin \left( \frac{\xi}{2}\right)\right|^{-H} \, d\xi.
\end{equation}

On the other side, for all $\mathbf{k} \in \Z^2$ we get, from the Definition \ref{def:fractscal} of the fractional scaling function and equation \eqref{prodscalfuncenfourier}
\begin{equation}\label{eqn:rosenblat2}
\int_{\R} \prod_{\ell=1}^2 \Phi_\Delta^{(\frac{H}{2})}(s-k_\ell) \, ds = \int_{- \frac{4\pi}{3}}^\frac{4\pi}{3} e^{i \xi (k_2-k_1)} \left(\frac{\sin(\xi/2)}{\xi/2} \right)^{H} \left| \widehat{\phi}(\xi) \right|^2 d\xi.
\end{equation}

Note that the two expressions \eqref{eqn:rosenblatt1} and \eqref{eqn:rosenblat2} are symmetric with respect to $k_1$ and $k_2$. Therefore, one can rewrite the recursive formula \eqref{formrec}, in that case,

\begin{align*}
 & s_{m+1,J}-s_{m,J} = 2^{-J H} \left(\left(Z_{J,m}^{(\frac{H}{2})} \right)^2-\mathbb{E}[\left(Z_{J,m}^{(\frac{H}{2})} \right)^2] \right) \int_{- \frac{4\pi}{3}}^\frac{4\pi}{3}  \left(\frac{\sin(\xi/2)}{\xi/2} \right)^{H} \left| \widehat{\phi}(\xi) \right|^2 d\xi \\
 & + 2 \times 2^{-J H} \!\!\!\!\!\! \!\!\!\!\!\!  \!\!\!\!\!\! \sum_{k=\max \{m_0, m- \lfloor 2^{\varepsilon J} \rfloor \}}^{m-1} \!\!\!\!\!\!  \left(Z_{J,m}^{(\frac{H}{2})}Z_{J,k}^{(\frac{H}{2})}-\mathbb{E}[Z_{J,m}^{(\frac{H}{2})}Z_{J,k}^{(\frac{H}{2})}] \right) \int_{- \frac{4\pi}{3}}^\frac{4\pi}{3} e^{i \xi (m-k)} \left(\frac{\sin(\xi/2)}{\xi/2} \right)^{H} \left| \widehat{\phi}(\xi) \right|^2 d\xi.
\end{align*}

Therefore, if $T >0$ is fixed and if
\begin{enumerate}
\item $\texttt{integralvector}$ is a vector of size $\lfloor 2^{\varepsilon J} \rfloor +1 $ such that, for all $0 \leq k \leq \lfloor 2^{\varepsilon J} \rfloor$, \texttt{integralvector}[k] is the value of
\[\int_{- \frac{4\pi}{3}}^\frac{4\pi}{3} e^{i \xi k} \left(\frac{\sin(\xi/2)}{\xi/2} \right)^{H} \left| \widehat{\phi}(\xi) \right|^2 d\xi; \]
\item $\texttt{IJ}$ is a vector containing $\mathcal{I}_J(T)$ with \texttt{IJ}[0]$=m_0$
\item $\texttt{farima}$ contains a simulation of the random variables $(Z_{J,\ell}^{(\frac{H}{2})})_{\ell \in \mathcal{I}_J(T)}$;
\item $\texttt{sigmaJ}(\texttt{farima},k_1,k_2,\frac{H}{2})$ is computing \eqref{eqn:rosenblatt1} using the simulation $\texttt{farima}$ and the formula \eqref{explicitexpressioncovrosenb}.
\item for all $m \in {I}_J(T)$, $\texttt{epaissi(m)}$ is a vector of size $\min \{\lfloor 2^{\varepsilon J} \rfloor,m-m_0 \}+1$ such that, for all $0 \leq i \leq \min \{\lfloor 2^{\varepsilon J} \rfloor,m-m_0 \}$, $\texttt{epaissi(m)}[i]=m-i$.
\end{enumerate}
the sequence $(s_{m,J})_{m \in \mathcal{I}_J(T)}$ is simulated via

\begin{codebox}
    \Procname{ \textsc{Rosenblatt}(J) }
    \li result[0]=\texttt{sigmaJ}(\texttt{farima},\texttt{IJ}[0],\texttt{IJ}[0],$\frac{H}{2}$) $\times$ \texttt{integralvector}[0]
    \li \For index=1 \To length(\texttt{IJ})
    \li \Do element=\texttt{epaissi}(\texttt{IJ}[index])
    \li  part=\texttt{sigmaJ}(\texttt{farima},element[0],element[0],$\frac{H}{2}$) $\times$ \texttt{integralvector}[0]
    \li  \For i=1 \To length(element)
    \li \Do part=part+2 $\times$ \texttt{sigmaJ}(\texttt{farima},element[0],element[i],$\frac{H}{2}$) $\times$ \texttt{integralvector}[i]
    \End 
    \li  result[index]=result[index-1]+part
    \End   
\end{codebox}

Indeed, it suffices to multiply the obtained vector result by $2^{-JH}$ to get a simulation of $(s_{m,J})_{m \in \mathcal{I}_J(T)}$.

The Figure \ref{fig:rosenblatt} below shows four simulated sample paths of Rosenblatt processes of respective Hurst parameters $0,6$; $0,7$; $0,8$ and $0,9$. They are obtained using our algorithm with parameters $J=20$, $a=0,75$ and $\varepsilon=10^{-4}$. Note that the FARIMA sequence $(Z_{J,\ell}^{(\frac{H}{2})})_{\ell \in \mathcal{I}_J(T)}$ is computed using fast wavelet transform as suggested in \cite{PIPIRAS200549,ABRY20062326}.

\begin{figure}[h!]
\begin{center}
\includegraphics[width=9.5cm]{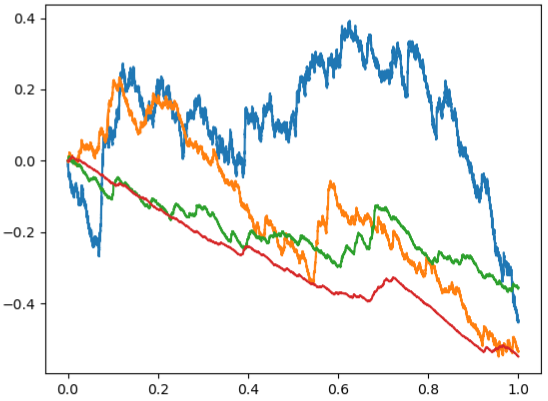}
\caption{Sample paths of the Rosenblatt process of Hurst parameter $0,6$ (blue), $0,7$ (orange), $0,8$ (green) and $0,9$ (red).} \label{fig:rosenblatt}
\end{center}
\end{figure}

\subsection{Simulation of the Hermite process of order 3}

The Hermite process of order 3 and Hurst parameter $H \in (1/2,1)$ is the process \eqref{eqn:def:ghp} with $d=3$ and $\mathbf{h}=(1+\frac{H-1}{3},1+\frac{H-1}{3},1+\frac{H-1}{3})$. To generate the random variables $(\sigma_{J,\textbf{k}}^{(\mathbf{h})})_{J \in \N, \mathbf{k} \in \Z^d}$, one can use the more explicit equation, derived from \eqref{eqn:defofsigma},

\begin{eqnarray}
&&\sigma_{J,k_1,k_2,k_3}^{(1+\frac{H-1}{3},1+\frac{H-1}{3},1+\frac{H-1}{3})} \nonumber\\
&&=Z_{J,k_{1}}^{(\frac{2H+1}{6})}Z_{J,k_{2}}^{(\frac{2H+1}{6})}Z_{J,k_{3}}^{(\frac{2H+1}{6})}-\mathbb{E}[Z_{J,k_{1}}^{(\frac{2H+1}{6})}Z_{J,k_{2}}^{(\frac{2H+1}{6})}]Z_{J,k_{3}}^{(\frac{2H+1}{6})} \nonumber\\
&&\hspace{0.5cm}-\mathbb{E}[Z_{J,k_{1}}^{(\frac{2H+1}{6})}Z_{J,k_{3}}^{(\frac{2H+1}{6})}]Z_{J,k_{2}}^{(\frac{2H+1}{6})}-\mathbb{E}[Z_{J,k_{2}}^{(\frac{2H+1}{6})}Z_{J,k_{3}}^{(\frac{2H+1}{6})}]Z_{J,k_{1}}^{(\frac{2H+1}{6})},\label{eqn:hermite1}
\end{eqnarray}
where $(Z_{J,\ell}^{(\frac{2H+1}{6})})_{\ell \in \N}$ is the Gaussian FARIMA $(0,\frac{2H+1}{6},0)$ sequence associated to the sequence of Gaussian random variables $(2^{\frac{J}{2}} \int_{\R} \phi(2^J x-k) dB(x))_{k\in\Z}$. In this case, we have \cite[Remark 3.5]{ayachehamonierloosveldt}, for all $J \in \N$ and $k_1,k_2 \in \Z$,
\begin{equation}\label{explicitcovher3}
\mathbb{E}[Z_{J,k_{1}}^{(\frac{2H+1}{6})}Z_{J,k_{2}}^{(\frac{2H+1}{6})}] = \frac{1}{2\pi} \int_0^{2 \pi} e^{i(k_2-k_1)} \left| 2 \sin \left( \frac{\xi}{2}\right)\right|^{-\frac{2H}{3}} \, d\xi.
\end{equation}
Also, from the Definition \ref{def:fractscal} of the fractional scaling function and equation \eqref{prodscalfuncenfourier}, we get, for all $\mathbf{k} \in \Z^3$,

\begin{align}\label{eqn:hermite2}
\int_{\R} & \prod_{\ell=1}^3 \Phi_\Delta^{(\frac{2H+1}{6})}(s-k_\ell) \, ds \nonumber \\ 
&= \iint_{[- \frac{4\pi}{3},\frac{4\pi}{3}]^2}  \!\!\!\!\!\! \!\!\!\!\!\! \!\!\!\!\!\! e^{i \xi (k_2-k_1)} e^{i \eta (k_3-k_2)}  \widehat{\phi}(\xi) \widehat{\phi}(\xi-\eta) \widehat{\phi}(\eta)  \left( \frac{\sin(\xi/2)}{\xi/2} \frac{\sin((\xi-\eta)/2)}{(\xi-\eta)/2} \frac{\sin(\eta/2)}{\eta/2} \right)^{\frac{2H+1}{6}}\!\!\!\!\!\! \!\!\!\!\!\! d\xi d\eta.
\end{align}

Using again the symmetry of the expressions \eqref{eqn:hermite1} and \eqref{eqn:hermite2}, we rewrite the recursive formula \eqref{formrec} in this context:
\begin{tiny}
\begin{align*}
 & s_{m+1,J}-s_{m,J} \\ &= 2^{-J H} \left[ \left( \left(Z_{J,m}^{(\frac{2H+1}{6})}\right)^3-3 \E\left[ \left(Z_{J,m}^{(\frac{2H+1}{6})}\right)^2 \right]\left(Z_{J,m}^{(\frac{2H+1}{6})}\right) \right) \right. \\ & \left. \iint_{[- \frac{4\pi}{3},\frac{4\pi}{3}]^2}    \widehat{\phi}(\xi) \widehat{\phi}(\xi-\eta) \widehat{\phi}(\eta)  \left( \frac{\sin(\xi/2)}{\xi/2} \frac{\sin((\xi-\eta)/2)}{(\xi-\eta)/2} \frac{\sin(\eta/2)}{\eta/2} \right)^{\frac{2H+1}{6}}\!\!\!\!\!\!  d\xi d\eta \right] \\
 &+ 3 \times 2^{-J H} \!\!\!\!\!\! \!\!\!\!\!\!  \!\!\!\!\!\! \sum_{k=\max \{m_0, m- \lfloor 2^{\varepsilon J} \rfloor \}}^{m-1} \left[ \left(\left(Z_{J,m}^{(\frac{2H+1}{6})}\right)^2 Z_{J,k}^{(\frac{2H+1}{6}) }-\E\left[\left(Z_{J,m}^{(\frac{2H+1}{6})}\right)^2  \right]Z_{J,k}^{(\frac{2H+1}{6})} - 2 \E \left[Z_{J,m}^{(\frac{2H+1}{6})} Z_{J,k}^{(\frac{2H+1}{6})}\right]Z_{J,m}^{(\frac{2H+1}{6})}\right) \right. \\
 & \left.\iint_{[- \frac{4\pi}{3},\frac{4\pi}{3}]^2}   e^{i \eta (m-k)} \widehat{\phi}(\xi) \widehat{\phi}(\xi-\eta) \widehat{\phi}(\eta)  \left( \frac{\sin(\xi/2)}{\xi/2} \frac{\sin((\xi-\eta)/2)}{(\xi-\eta)/2} \frac{\sin(\eta/2)}{\eta/2} \right)^{\frac{2H+1}{6}}\!\!\!\!\!\!  d\xi d\eta \right] \\
 &+ 3 \times 2^{-J H} \!\!\!\!\!\! \!\!\!\!\!\!  \!\!\!\!\!\! \sum_{k=\max \{m_0, m- \lfloor 2^{\varepsilon J} \rfloor \}}^{m-1} \left[ \left(\left(Z_{J,k}^{(\frac{2H+1}{6})}\right)^2 Z_{J,m}^{(\frac{2H+1}{6}) }-\E\left[\left(Z_{J,k}^{(\frac{2H+1}{6})}\right)^2  \right]Z_{J,m}^{(\frac{2H+1}{6})} - 2 \E \left[Z_{J,k}^{(\frac{2H+1}{6})} Z_{J,m}^{(\frac{2H+1}{6})}\right]Z_{J,k}^{(\frac{2H+1}{6})}\right) \right. \\  & \left.\iint_{[- \frac{4\pi}{3},\frac{4\pi}{3}]^2}   e^{i \eta (m-k)} \widehat{\phi}(\xi) \widehat{\phi}(\xi-\eta) \widehat{\phi}(\eta)  \left( \frac{\sin(\xi/2)}{\xi/2} \frac{\sin((\xi-\eta)/2)}{(\xi-\eta)/2} \frac{\sin(\eta/2)}{\eta/2} \right)^{\frac{2H+1}{6}}\!\!\!\!\!\!  d\xi d\eta \right] \\
 & + 6 \times 2^{-J H} \!\!\!\!\!\! \!\!\!\!\!\!  \!\!\!\!\!\! \sum_{k=\max \{m_0, m- \lfloor 2^{\varepsilon J} \rfloor \}}^{m-1} \sum_{\ell=\max \{m_0, m- \lfloor 2^{\varepsilon J} \rfloor \}}^{k-1} \left[ \left(Z_{J,m}^{(\frac{2H+1}{6})}Z_{J,k}^{(\frac{2H+1}{6})}Z_{J,\ell}^{(\frac{2H+1}{6})}-\mathbb{E}\left[Z_{J,m}^{(\frac{2H+1}{6})}Z_{J,k}^{(\frac{2H+1}{6})}\right]Z_{J,\ell}^{(\frac{2H+1}{6})} \right.\right. \\
&  - \left.\mathbb{E}\left[Z_{J,m}^{(\frac{2H+1}{6})}Z_{J,\ell}^{(\frac{2H+1}{6})}\right]Z_{J,k}^{(\frac{2H+1}{6})}-\mathbb{E}\left[Z_{J,k}^{(\frac{2H+1}{6})}Z_{J,\ell}^{(\frac{2H+1}{6})}\right]Z_{J,m}^{(\frac{2H+1}{6})}  \right) \iint_{[- \frac{4\pi}{3},\frac{4\pi}{3}]^2}   e^{i \xi (m-k)} e^{i \eta (k-\ell)} \\ & \left.\widehat{\phi}(\xi) \widehat{\phi}(\xi-\eta) \widehat{\phi}(\eta)  \left( \frac{\sin(\xi/2)}{\xi/2} \frac{\sin((\xi-\eta)/2)}{(\xi-\eta)/2} \frac{\sin(\eta/2)}{\eta/2} \right)^{\frac{2H+1}{6}}\!\!\!\!\!\!  d\xi d\eta\right].
\end{align*}
\end{tiny}

Therefore, if $T >0$ is fixed and if
\begin{enumerate}
\item $\texttt{integralmatrix}$ is a matrix of dimension $\lfloor 2^{\varepsilon J} \rfloor +1 $ such that, for all $0 \leq k,\ell \leq \lfloor 2^{\varepsilon J} \rfloor$, \texttt{integralmatrix}[k,l] is the value of
\[\iint_{[- \frac{4\pi}{3},\frac{4\pi}{3}]^2}  \!\!\!\!\!\! \!\!\!\!\!\! \!\!\!\!\!\! e^{i \xi k} e^{i \eta \ell}  \widehat{\phi}(\xi) \widehat{\phi}(\xi-\eta) \widehat{\phi}(\eta)  \left( \frac{\sin(\xi/2)}{\xi/2} \frac{\sin((\xi-\eta)/2)}{(\xi-\eta)/2} \frac{\sin(\eta/2)}{\eta/2} \right)^{\frac{2H+1}{6}}\!\!\!\!\!\! \!\!\!\!\!\! d\xi d\eta \]

\item $\texttt{IJ}$ is a vector containing $\mathcal{I}_J(T)$ with \texttt{IJ}[0]$=m_0$
\item $\texttt{farima}$ contains a simulation of the random variables $(Z_{J,\ell}^{(\frac{2H+1}{6})})_{\ell \in \mathcal{I}_J(T)}$;
\item $\texttt{sigmaJ}(\texttt{farima},k_1,k_2,k_3,\frac{H}{2})$ is computing \eqref{eqn:hermite1} using the simulation $\texttt{farima}$ and the formula \eqref{explicitcovher3}.
\item for all $m \in {I}_J(T)$, $\texttt{epaissi(m)}$ is a vector of size $\min \{\lfloor 2^{\varepsilon J} \rfloor,m-m_0 \}+1$ such that, for all $0 \leq i \leq \min \{\lfloor 2^{\varepsilon J} \rfloor,m-m_0 \}$, $\texttt{epaissi(m)}[i]=m-i$.
\end{enumerate}
the sequence $(s_{m,J})_{m \in \mathcal{I}_J(T)}$ is simulated via

\begin{small}
\begin{codebox}
    \Procname{ \textsc{Hermite3}(J) }
    \li result[0]=\texttt{sigmaJ}(\texttt{farima},\texttt{IJ}[0],\texttt{IJ}[0],\texttt{IJ}[0],$\frac{H}{2}$) $\times$ \texttt{integralmatrix}[0,0]
    \li \For index=1 \To length(\texttt{IJ})
    \li \Do element=\texttt{epaissi}(\texttt{IJ}[index])
    \li  part=\texttt{sigmaJ}(\texttt{farima},element[0],element[0],element[0],$\frac{H}{2}$) $\times$ \texttt{integralmatrix}[0,0]
    \li  \For i=1 \To length(element)
    \li \Do part=part+3 $\times$ \texttt{sigmaJ}(\texttt{farima},element[0],element[0],element[i],$\frac{H}{2}$) $\times$ \texttt{integralmatrix}[0,i]
    \li  part=part+3 $\times$ \texttt{sigmaJ}(\texttt{farima},element[0],element[i],element[i],$\frac{H}{2}$) $\times$ \texttt{integralmatrix}[0,i]
    \li \For j=i+1 \To  length(element)
    \li \Do part=part+6 $\times$ \texttt{sigmaJ}(\texttt{farima},element[0],element[i],element[j],$\frac{H}{2}$) $\times$ \texttt{integralmatrix}[i,j-i]  
    \End
    \End
    \li  result[index]=result[index-1]+part
    \End   
\end{codebox}

\end{small}

Indeed, it suffices to multiply the obtained vector result by $2^{-JH}$ to get a simulation of $(s_{m,J})_{m \in \mathcal{I}_J(T)}$.

The Figure \ref{fig:hermite} below shows four simulated sample pathes of Hermite processes of order $3$ and respective Hurst parameter $0,6$; $0,7$; $0,8$ and $0,9$. They are obtained using our algorithm again with parameters $J=20$, $a=0,75$ and $\varepsilon=10^{-4}$. The FARIMA sequence $(Z_{J,\ell}^{(\frac{2H+1}{6})})_{\ell \in \mathcal{I}_J(T)}$ is once again computed via a fast wavelet transform.

\begin{figure}[h!]
\begin{center}
\includegraphics[width=9.5cm]{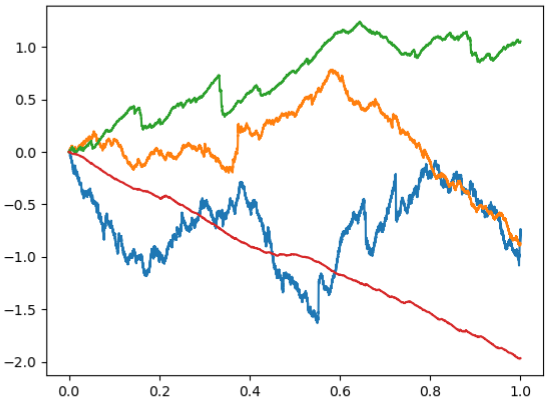}
\caption{Sample paths of the Hermite process of order $3$ of Hurst parameter $0,6$ (blue), $0,7$ (orange), $0,8$ (green) and $0,9$ (red).} \label{fig:hermite}
\end{center}
\end{figure}

Of course, our algorithm can also be used to simulate sample paths of the Fractional Brownian motion with Hurst parameter $H \in (1/2,1)$. In this context, the integrals appearing both in the definitions of $\{S_{\mathbf{h},J}^{(d)}(t)\}_{t \in \R_+}$ and $\{\widetilde{S}_{\mathbf{h},J}^{(d)}(t)\}_{t \in \R_+}$ are all equal to $1$. Thus, the approximation scheme, consisting in summing terms of the Gaussian FARIMA sequence $(Z_{J,\ell}^{(H-\frac{1}{2})})_{\ell \in \N}$, according to the sets $(\mathcal{J}_J^1(m2^{-J}+2^{-aJ}))_{m \in \mathcal{I}_J}$ is particularly elementary to implement. In the figures below, we compare simulated sample paths of the Fractional Brownian, the Rosenblatt process and the Hermite process of order $3$, with Hurst parameter $0,6$; $0,7$; $0,8$ and $0,9$. All simulations are once again obtained using our algorithm with parameters $J=20$, $a=0,75$ and $\varepsilon=10^{-4}$.

\begin{figure}[h!]
\begin{center}
\includegraphics[width=9.5cm]{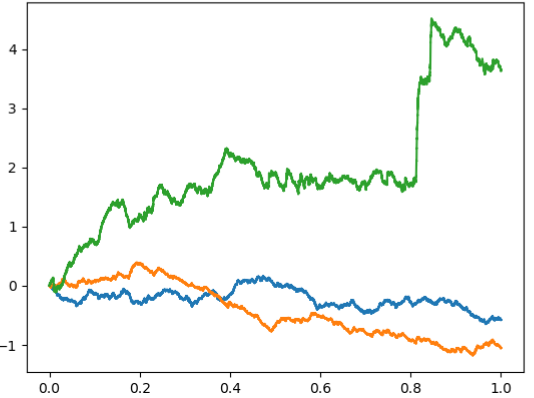}
\caption{Sample paths of the Fractional Brownian Motion (blue), the Rosenblatt process (orange) and the Hermite process of order $3$ (green) with Hurst parameter $0,6$.} \label{fig:0.6}
\end{center}
\end{figure}

\begin{figure}[h!]
\begin{center}
\includegraphics[width=9.5cm]{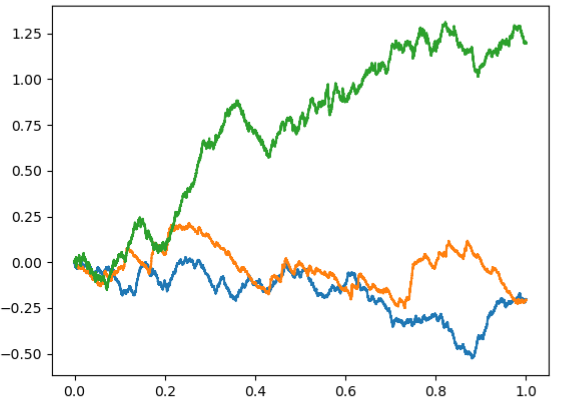}
\caption{Sample paths of the Fractional Brownian Motion (blue), the Rosenblatt process (orange) and the Hermite process of order $3$ (green) with Hurst parameter $0,7$.} \label{fig:0.7}
\end{center}
\end{figure}

\begin{figure}[h!]
\begin{center}
\includegraphics[width=9.5cm]{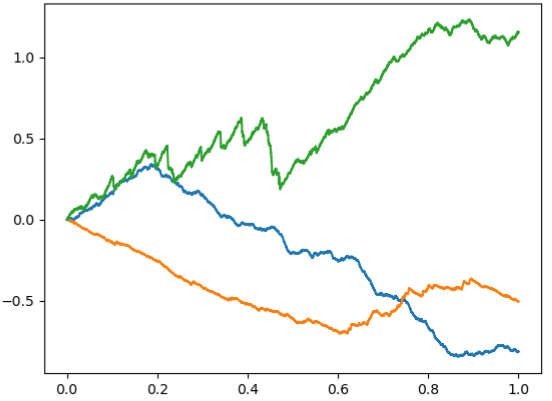}
\caption{Sample paths of the Fractional Brownian Motion (blue), the Rosenblatt process (orange) and the Hermite process of order $3$ (green) with Hurst parameter $0,8$.} \label{fig:0.8}
\end{center}
\end{figure}

\begin{figure}[h!]
\begin{center}
\includegraphics[width=9.5cm]{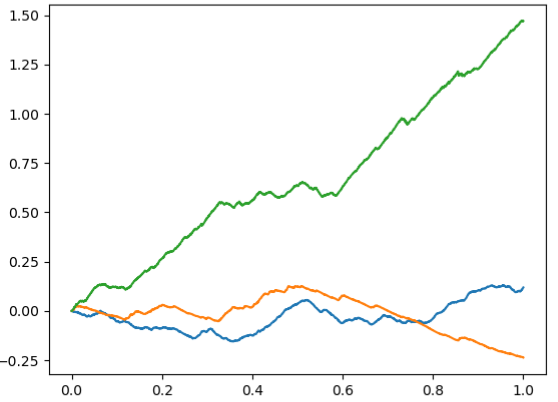}
\caption{Sample paths of the Fractional Brownian Motion (blue), the Rosenblatt process (orange) and the Hermite process of order $3$ (green) with Hurst parameter $0,9$.} \label{fig:0.9}
\end{center}
\end{figure}

\begin{Rmk}
To get the simulations in Figures \ref{fig:0.6}, \ref{fig:0.7}, \ref{fig:0.8}, \ref{fig:0.9}, we use the standard normalisation of Hermite processes (see Remark \ref{rmk:normalisation} above) which is more appropriate to compare the sample paths.
\end{Rmk}

\subsection{Simulation of the generalized Hermite process of order 3}

Our approach can also be used to numerically simulate generalized Hermite processes. We consider here the case of a process of order $d=3$ and arbitrary $\mathbf{h}=(h_1,h_2,h_3)$ satisfying conditions \eqref{eqn:hermi:cond}. To generate the random variables $(\sigma_{J,\textbf{k}}^{(\mathbf{h})})_{J \in \N, \mathbf{k} \in \Z^d}$, we use the more explicit equation, derived from \eqref{eqn:defofsigma},
\begin{eqnarray}
&&\sigma_{J,k_1,k_2,k_3}^{(h_1,h_2,h_3)} \nonumber\\
&&=Z_{J,k_{1}}^{(h_{1}-1/2)}Z_{J,k_{2}}^{(h_{2}-1/2)}Z_{J,k_{3}}^{(h_{3}-1/2)}-\mathbb{E}[Z_{J,k_{1}}^{(h_{1}-1/2)}Z_{J,k_{2}}^{(h_{2}-1/2)}]Z_{J,k_{3}}^{(h_{3}-1/2)} \nonumber \\
&&\hspace{0.5cm}-\mathbb{E}[Z_{J,k_{1}}^{(h_{1}-1/2)}Z_{J,k_{3}}^{(h_{3}-1/2)}]Z_{J,k_{2}}^{(h_{2}-1/2)}-\mathbb{E}[Z_{J,k_{2}}^{(h_{2}-1/2)}Z_{J,k_{3}}^{(h_{3}-1/2)}]Z_{J,k_{1}}^{(h_{1}-1/2)}. \label{genfarimaorder3}
\end{eqnarray}
where, for any $\ell \in \{1,2,3\}$, $(Z_{J,\ell}^{(h_\ell-1/2)})_{\ell \in \N}$ is the Gaussian FARIMA $(0,h_{\ell}-1/2,0)$ sequence associated to the sequence of Gaussian random variables $(2^{\frac{J}{2}} \int_{\R} \phi(2^J x-k) dB(x))_{k\in\Z}$. In this case, we use again \cite[Remark 3.5]{ayachehamonierloosveldt}, to compute, for all $J \in \N$, $\ell,m \in \{1,2,3\}$  and $k_\ell,k_m \in \Z$,
\begin{align}\label{explicitcovgenher}
\E[Z_{j,k_\ell}^{(h_\ell-1/2)}Z_{j,k_m}^{(h_m-1/2)}] =\frac{1}{2\pi}\int_0^{2\pi} e^{i(k_\ell-k_m)\xi} (1-e^{-i\xi})^{-(h_\ell-1/2)} (1-e^{i\xi})^{-(h_m-1/2)} \, d\xi.
\end{align}
Also, from the Definition \ref{def:fractscal} of the fractional scaling function and equation \eqref{prodscalfuncenfourier}, we get, for all $\mathbf{k} \in \Z^3$,

\begin{align}\label{eqn:genhermite2}
\int_{\R} & \prod_{\ell=1}^3 \Phi_\Delta^{(h_\ell-1/2}(s-k_\ell) \, ds \nonumber \\ 
&= \iint_{[- \frac{4\pi}{3},\frac{4\pi}{3}]^2}  \!\!\!\!\!\! \!\!\!\!\!\! \!\!\!\!\!\! e^{i \xi (k_2-k_1+h_2-h_1)} e^{i \eta (k_3-k_2+h_3-h_2)}  \widehat{\phi}(\xi) \widehat{\phi}(\xi-\eta) \widehat{\phi}(\eta)\left( \frac{\sin(\xi/2)}{\xi/2} \right)^{(h_1-1/2)} \times \cdots  \nonumber\\ & \cdots \times  \left( \frac{\sin((\xi-\eta)/2)}{(\xi-\eta)/2} \right)^{(h_2-1/2)}  \left(\frac{\sin(\eta/2)}{\eta/2} \right)^{(h_3-1/2)} \!\!\!\!\!\! \!\!\!\!\!\! d\xi d\eta.
\end{align}

Now, the expressions \eqref{genfarimaorder3} and \eqref{eqn:genhermite2} are not symmetric, which means that we have to compute much more quantities in our algorithm. Indeed, if $T >0$ is fixed and if
\begin{enumerate}
\item $\texttt{integralmatrix}$ is a matrix of dimension $2 \times \lfloor 2^{\varepsilon J} \rfloor +1 $ such that, for all $- \lfloor 2^{\varepsilon J} \rfloor  \leq k,\ell \leq \lfloor 2^{\varepsilon J} \rfloor$, \texttt{integralmatrix}[k,l] is the value of
\begin{align*}
 \iint_{[- \frac{4\pi}{3},\frac{4\pi}{3}]^2} & \!\!\!\!\!\! \!\!\!\!\!\! \!\!\!\!\!\! e^{i \xi (k+h_2-h_1)} e^{i \eta (\ell+h_3-h_2)}  \widehat{\phi}(\xi) \widehat{\phi}(\xi-\eta) \widehat{\phi}(\eta)\left( \frac{\sin(\xi/2)}{\xi/2} \right)^{(h_1-1/2)} \times \cdots  \nonumber\\ & \cdots \times  \left( \frac{\sin((\xi-\eta)/2)}{(\xi-\eta)/2} \right)^{(h_2-1/2)}  \left(\frac{\sin(\eta/2)}{\eta/2} \right)^{(h_3-1/2)} \!\!\!\!\!\! \!\!\!\!\!\! d\xi d\eta.
\end{align*}

\item $\texttt{IJ}$ is a vector containing $\mathcal{I}_J(T)$ with \texttt{IJ}[0]$=m_0$
\item for any $\ell \in \{1,2,3\}$, $f_\ell$ contains a simulation of the random variables $(Z_{J,m}^{(h_\ell-1/2)})_{m \in \mathcal{I}_J(T)}$;
\item $\texttt{sigmaJ}(f_1,f_2,f_3,k_1,k_2,k_3,\mathbf{h})$ is computing \eqref{genfarimaorder3} using the three simulations $f_1$, $f_2$, $f_3$ and the formula \eqref{explicitcovgenher}.
\item for all $m \in {I}_J(T)$, $\texttt{epaissi(m)}$ is a vector of size $\min \{\lfloor 2^{\varepsilon J} \rfloor,m-m_0 \}+1$ such that, for all $0 \leq i \leq \min \{\lfloor 2^{\varepsilon J} \rfloor,m-m_0 \}$, $\texttt{epaissi(m)}[i]=m-i$.
\end{enumerate}
the sequence $(s_{m,J})_{m \in \mathcal{I}_J(T)}$ is simulated via

\begin{small}
\begin{codebox}
    \Procname{ \textsc{GenHermite3}(J) }
    \li result[0]=\texttt{sigmaJ}($f_1$,$f_2$,$f_3$,\texttt{IJ}[0],\texttt{IJ}[0],\texttt{IJ}[0],$\mathbf{h}$) $\times$ \texttt{integralmatrix}[0,0]
    \li \For index=1 \To length(\texttt{IJ})
    \li \Do element=\texttt{epaissi}(\texttt{IJ}[index])
    \li  part=\texttt{sigmaJ}($f_1$,$f_2$,$f_3$,element[0],element[0],element[0],$\mathbf{h}$) $\times$ \texttt{integralmatrix}[0,0]
    \li  \For i=1 \To length(element)
    \li \Do part=part+ \texttt{sigmaJ}($f_1$,$f_2$,$f_3$,element[0],element[0],element[i],$\mathbf{h}$) $\times$ \texttt{integralmatrix}[0,-i]
    \li part=part+ \texttt{sigmaJ}($f_1$,$f_2$,$f_3$,element[0],element[i],element[0],$\mathbf{h}$) $\times$ \texttt{integralmatrix}[-i,i]
    \li part=part+ \texttt{sigmaJ}($f_1$,$f_2$,$f_3$,element[i],element[0],element[0],$\mathbf{h}$) $\times$ \texttt{integralmatrix}[i,0]
    \li  part=part+ \texttt{sigmaJ}($f_1$,$f_2$,$f_3$,element[0],element[i],element[i],$\mathbf{h}$) $\times$ \texttt{integralmatrix}[-i,0]
    \li  part=part+ \texttt{sigmaJ}($f_1$,$f_2$,$f_3$,element[i],element[0],element[i],$\mathbf{h}$) $\times$ \texttt{integralmatrix}[i,-i]
    \li part=part+ \texttt{sigmaJ}($f_1$,$f_2$,$f_3$,element[i],element[i],element[0],$\mathbf{h}$) $\times$ \texttt{integralmatrix}[0,i]
    \li \For j=1 \To  length(element)
    \li \Do \If j $\neq$ i
    \li \Do part=part+ \texttt{sigmaJ}($f_1$,$f_2$,$f_3$,element[0],element[i],element[j],$\mathbf{h}$) $\times$ \texttt{integralmatrix}[-i,j-i]
    \li part=part+ \texttt{sigmaJ}($f_1$,$f_2$,$f_3$,element[0],element[j],element[i],$\mathbf{h}$) $\times$ \texttt{integralmatrix}[-j,j-i]
    \li part=part+ \texttt{sigmaJ}($f_1$,$f_2$,$f_3$,element[i],element[0],element[j],$\mathbf{h}$) $\times$ \texttt{integralmatrix}[i,-j]
    \li part=part+ \texttt{sigmaJ}($f_1$,$f_2$,$f_3$,element[j],element[0],element[i],$\mathbf{h}$) $\times$ \texttt{integralmatrix}[j,-i]
    \li part=part+ \texttt{sigmaJ}($f_1$,$f_2$,$f_3$,element[i],element[j],element[0],$\mathbf{h}$) $\times$ \texttt{integralmatrix}[i-j,j]  
    \li part=part+ \texttt{sigmaJ}($f_1$,$f_2$,$f_3$,element[j],element[i],element[0],$\mathbf{h}$) $\times$ \texttt{integralmatrix}[j-i,i]      
    \End
    \End
    \End
    \li  result[index]=result[index-1]+part
    \End   
\end{codebox}

\end{small}

Due to the numerous additional computations, this algorithm clearly needed more time to produce a simulation than its counterpart for the usual Hermite process of order $3$. For this reason, the sample path displayed in Figure \ref{fig:genhermite3} below is obtained with $J=15$ (instead of $J=20$ in the previous figures). The other parameters are $h_1=0,8$; $h_2=0,85$ and $h_3=0,9$, $a=0,75$ and $\varepsilon=10^{-4}$.

\begin{figure}[h!]
\begin{center}
\includegraphics[width=9.5cm]{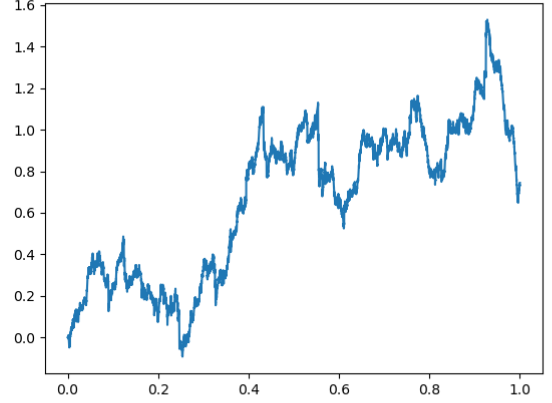}
\caption{Sample path of the generalized Hermite of order $3$ with index $\mathbf{h}=(0,8 , 0,85, 0,9)$.} \label{fig:genhermite3}
\end{center}
\end{figure}

\clearpage

\bibliography{biblio}{}
\bibliographystyle{plain}

\end{document}